\documentclass[11pt]{article}

\usepackage{natbib}
\usepackage{amsmath, amsthm, amssymb}
\usepackage{graphicx,color}
\usepackage{amsfonts}
\usepackage{relsize}
\numberwithin{equation}{section}

\newcounter{prop}
\newenvironment{prop}[1][]{\refstepcounter{prop}\par\medskip%
   \textit{Proposition~\theprop. #1} \rmfamily}{}
   
\newcounter{thm}
\newenvironment{thm}[1][]{\refstepcounter{thm}\par\medskip%
   \textit{Theorem~\thethm. #1} \rmfamily}{}
   
\newcounter{corollary}

\newcounter{lemma}
\newenvironment{lemma}[1][]{\refstepcounter{lemma}\par\medskip%
   \textit{Lemma~\thelemma. #1} \rmfamily}{}

\newcounter{example}
\newenvironment{example}[1][]{\refstepcounter{example}\par\medskip%
   \textit{Example~\theexample. #1} \rmfamily}{}
   
\newcounter{remark}
\newenvironment{remark}[1][]{\refstepcounter{remark}\par\medskip%
   \textit{Remark~\theremark. #1} \rmfamily}{}

\def \var{\mbox{Var}}
\def \cov{\mbox{Cov}}

\def \asconv {\stackrel{a.s.}{\longrightarrow}}
\def \indist {\stackrel{\mathcal{L}}{\longrightarrow}}
\def \sumn {\displaystyle\sum_{t=1}^n}

\def \var{\mbox{Var}}

\def \E{\mbox{E}}
\def \xx{\mathbf{x}}
\def \yy{\mathbf{y}}
\def \XX{\mathbf{X}}
\def \argmin{\mbox{argmin}}
\def \argmax{\mbox{argmax}}

\begin{document}
\title{Theory and Inference for a Class of Observation-Driven Models with Application to Time Series of Counts}
\author{Richard A. Davis and Heng Liu\footnote{Corresponding author: Department of Statistics, Columbia University, 1255 Amsterdam Avenue, MC4690, New York, NY 10027, USA; Email: hengliu@stat.columbia.edu}\\~\\
Columbia University}

\date{}
\maketitle

\begin{abstract}
This paper studies theory and inference related to a class of time series models that incorporates nonlinear dynamics. It is assumed that the observations follow a one-parameter exponential family of distributions given an accompanying process that evolves as a function of lagged observations.  We employ an iterated random function approach and a special coupling technique to show that, under suitable conditions on the parameter space,  the conditional mean process is a geometric moment contracting Markov chain and that the observation process is absolutely regular with geometrically decaying coefficients. Moreover the asymptotic theory of the maximum likelihood estimates of the parameters is established under some mild assumptions. These models are applied to two examples; the first is the number of transactions per minute of Ericsson stock and the second is related to return times of extreme events of Goldman Sachs Group stock.

\noindent\small{\textbf{Keywords:}~Absolute regularity; Ergodicity; Geometric moment contraction; Iterated random functions; One-parameter exponential family; Time series of counts}
\end{abstract}
\normalsize

\section{Introduction}

With a surge in the range of applications from economics, finance, environmental science, social science and epidemiology, there has been renewed interest in developing models for time series of counts. The majority of these models assume that the observations follow a Poisson distribution conditioned on an accompanying intensity process that drives the dynamics of the models, e.g., \cite{Davis03}, \cite{Fokianos}, \cite{Neumann}, \cite{Sarah} and \cite{WeakDepPoisAR}. According to whether the evolution of the intensity process depends on the observations or solely on an external process, \cite{Cox81} classified the models into observation-driven and parameter-driven. This paper focuses on the theory and inference for a particular class of observation-driven models.

Many of the commonly used models, such as the Poisson integer-valued GARCH (INGARCH), are special cases of our model. For an INGARCH, the observations $\{Y_t\}$ given the intensity process $\{\lambda_t\}$ follow a Poisson distribution and $\lambda_t$ is a linear combination of its lagged values and lagged $Y_t$. The model is capable of capturing positive temporal correlation in the observations and it is relatively easy to fit via maximum likelihood. \cite{Ferland} showed the second moment stationarity through a sequence of approximating processes and \cite{Fokianos} established the consistency and asymptotic normality of the MLE by introducing a perturbed model. However, all the above results rely heavily on the Poisson assumption and the GARCH-like dynamics of $\lambda_t$. Later \cite{Neumann} relaxed the linear assumption to a general contracting evolution rule and proved the absolute regularity for this Poisson count process and \cite{WeakDepPoisAR} showed the existence of moments under similar conditions by utilizing the concept of weak dependence.

In our study the conditional distribution of the observation $Y_t$ given the past is assumed to follow a one-parameter exponential family. The temporal dependence in the model is defined through recursions relating the conditional mean process $X_t$ with its lagged values and lagged observations. Theory from iterated random functions (IRF), see e.g., \cite{Diaconis} and \cite{Weibiao04}, is utilized to establish some key stability properties, such as existence of a stationary and mixing solution. This theory allows us to consider both linear and nonlinear dynamic models as well as inference questions. In particular, the asymptotic normality of the maximum likelihood estimates can be established. The nonlinear dynamic models are also investigated in a simulation study and both linear and nonlinear models are applied to two real datasets.

The organization of the paper is as follows. Section 2 formulates the model and establishes stability properties. The maximum likelihood estimates of the parameters and the relevant asymptotic theory are derived in Section 3. Examples of both linear and nonlinear dynamic models are considered in Section 4. Numerical results, including a simulation study and two data applications are given in Section 5, where the models are applied to the number of transactions per minute of Ericsson stock and to the return times of extreme events of Goldman Sachs Group (GS) stock. Some diagnostic tools for assessing and comparing model performance are also given in Section 5. Appendix A reviews some standard properties of the one-parameter exponential family and the proofs of the key results in Sections 2-4 are deferred to Appendix B.

\section{Model formulation and stability properties}
\subsection{One-parameter exponential family}
A random variable $Y$ is said to follow a distribution of the one-parameter exponential family if its probability density function with respect to some $\sigma$-finite measure $\mu$ is given by
\begin{eqnarray}
p(y|\eta)=\exp\{\eta y-A(\eta)\}h(y), ~~~ y\ge 0, \label{eq:expfamily}
\end{eqnarray}
where $\eta$ is the natural parameter, and $A(\eta)$ and $h(y)$ are known functions. If $B(\eta)=A'(\eta)$, then it is known that $\E Y=B(\eta)$ and $\var(Y)=B'(\eta)$. The derivative of $A(\eta)$ exists generally for the exponential family, see e.g., \cite{TPE}. Since $B'(\eta)=\var(Y)>0$, so $B(\eta)$ is strictly increasing, which establishes a one-to-one association between the values of $\eta$ and $B(\eta)$. Moreover, because we assume that the support of $Y$ is non-negative throughout this paper, so $B(\eta)=\E Y>0$, which implies that $A(\eta)$ is strictly increasing. Other properties of this family of distributions are presented in Appendix A.

Many familiar distributions belong to this family, including Poisson, negative binomial, Bernoulli, exponential, etc. If the shape parameter is fixed, then the gamma distribution is also a member of this family. While we restrict consideration to only the univariate case, extensions to the multi-parameter exponential family is a topic of future research.

\subsection{Model formulation}

Set $\mathcal{F}_0=\sigma\{\eta_1\}$, where $\eta_1$ is a natural parameter of (\ref{eq:expfamily}) and assumed fixed for the moment. Let $Y_1, Y_2, \ldots$ be observations from a model that is defined recursively in the following fashion,
\begin{eqnarray}
Y_t|\mathcal{F}_{t-1}\sim p(y|\eta_t),~~~X_t=g_{\theta}(X_{t-1}, Y_{t-1}), \label{eq:expmodel}
\end{eqnarray}
for all $t\ge 1$, where $p(y|\eta_t)$ is defined in (\ref{eq:expfamily}), $\mathcal{F}_t=\sigma\{\eta_1, Y_1,\ldots, Y_t\}$ and $X_t$ is the conditional mean process, i.e., $X_t=B(\eta_t)=\E(Y_t|\mathcal{F}_{t-1})$. Here $g_{\theta}(x, y)$ is a non-negative bivariate function defined on $[0, \infty)\times [0,\infty)$ when $Y_t$ has a continuous conditional distribution or on $[0,\infty)\times \mathbb{N}_0$, where $\mathbb{N}_0=\{0, 1,\ldots\}$, when $Y_t$ only takes non-negative integers. Throughout, we assume that the function $g_{\theta}$ satisfies a contraction condition, i.e., for any $x,x'\ge 0$, and $y,y'\in [0,\infty)~\mbox{or}~\mathbb{N}_0$, 
\begin{eqnarray}
|g_{\theta}(x,y)-g_{\theta}(x',y')|\le a|x-x'|+b|y-y'|, \label{ContractionFunction}
\end{eqnarray} 
where $a$ and $b$ are non-negative constants with $a+b<1$. Note that (\ref{ContractionFunction}) implies 
\begin{eqnarray}
g_{\theta}(x,y)\le g_{\theta}(0,0)+ax+by, ~~\mbox{for any}~~x,y\ge 0. \label{eq:BoundOfG}
\end{eqnarray}
We point out that model (\ref{eq:expmodel}) with the function $g_{\theta}$ satisfying (\ref{ContractionFunction}) includes the Poisson INGARCH model (see Example \ref{PoissonIngarchExample}) and the exponential autoregressive model (\ref{eq:PoisExpModel}) as special cases under some restrictions on the parameter space. The generalized linear autoregressive moving average model (GLARMA) (see \cite{Davis03}) also belongs to this class, although the contraction condition is not necessarily satisfied. Only under very simple model specifications have the stability properties of GLARMA been established and the relevant work is still ongoing. The primary focus of this paper is on the conditional mean process $\{X_t\}$, which can be easily seen as a time-homogeneous Markov chain. Note that the observation process $\{Y_t\}$ is not a Markov chain itself.

\subsection{Strict stationarity} 
The iterated random function approach (see e.g., \cite{Diaconis} and \cite{Weibiao04}) provides a useful tool when investigating the stability properties of Markov chains and turns out to be particularly instrumental in our research. In the definition of iterated random functions (IRF), the state space $(\mathcal{W}, \rho)$ is assumed to be a complete and separable metric space. Then a sequence of \emph{iterated random functions} $\{f_{\theta_t}\}$ is defined through
\begin{eqnarray*}
W_t=f_{\theta_t}(W_{t-1}),~~ t\in \mathbb{N},
\end{eqnarray*}
where $\{\theta_t\}_{t\ge 1}$ take values in another measurable space $\Theta$ and are independently distributed with identical marginal distribution, and $W_0$ is independent of $\{\theta_t\}_{t\ge 1}$.

In working with iterated random functions, \cite{Weibiao04} introduces the idea of geometric moment contraction (GMC), which is useful for deriving further properties of IRF. Our research is also relying heavily on GMC. Suppose there exists a stationary solution to the Markov chain $\{W_t\}$, denoted by $\varpi$, let $W_0, W_0'\sim \varpi$ be independent of each other and of $\{\theta_t\}_{t\ge 1}$, and define $W_t(w)=f_{\theta_t}\circ f_{\theta_{t-1}}\circ\ldots\circ f_{\theta_1}(w)$. Then $\{W_t\}$ is said to be \emph{geometric moment contracting} if there exist an $\alpha>0$, a $C=C(\alpha)>0$ and an $r=r(\alpha)\in (0,1)$ such that, for all $t\in \mathbb{N}$,
\begin{eqnarray*}
\E\{\rho^{\alpha}(W_n(W_0), W_n(W_0'))\}\le Cr^n.
\end{eqnarray*}
The conditional mean process $\{X_t\}$ specified in (\ref{eq:expmodel}) can be embedded into the framework of IRF and shown to be GMC.

In this section and the next we use $g$ to represent the function $g_{\theta}$ in (\ref{eq:expmodel}) evaluated at the true parameter. For any $u\in (0, 1)$, the random function $f_{u}(x)$ is defined as
\begin{eqnarray}
f_{u}(x):=g\bigr(x, F^{-1}_{x}(u)\bigr), \label{eq:IRFexp}
\end{eqnarray}
where $F_x$ is the cumulative distribution function of $p(y|\eta)$ in (\ref{eq:expfamily}) with $x=B(\eta)$, and its inverse $F_x^{-1}(u):=\inf\{t\ge 0: F_x(t)\ge u\}$ for $u\in [0,1]$. Let $\{U_t\}$ be a sequence of independent and identically distributed (iid) uniform $(0,1)$ random variables, then the Markov chain $\{X_t\}$ defined in (\ref{eq:expmodel}) starting from $X_0=x$ can be represented as the so-called forward process $X_t(x)=(f_{U_t}\circ f_{U_{t-1}}\circ\ldots\circ f_{U_1})(x)$. The corresponding backward process is defined as $Z_t(x)=(f_{U_1}\circ f_{U_2}\circ\ldots\circ f_{U_t})(x)$, which has the same distribution as $X_t(x)$ for any $t$.
\begin{prop}
\label{modelgmc}
Assume model (\ref{eq:expmodel}) and that the function $g$ satisfies the contraction condition (\ref{ContractionFunction}). Then 
\begin{enumerate}
\item There exists a random variable $Z_{\infty}$ such that, for all $x\in S$, $Z_n(x)\rightarrow Z_{\infty}$ almost surely. The limit $Z_{\infty}$ does not depend on $x$ and has distribution $\pi$, which is the stationary distribution of $\{X_t\}$.
\item The Markov chain $\{X_t, t\ge 1\}$ is geometric moment contracting with $\pi$ as its unique stationary distribution. In addition, $\E_{\pi}X_1<\infty$. 
\item If $\{X_t, t\ge 1\}$ starts from $\pi$, i.e., $X_1\sim \pi$, then $\{Y_t, t\ge 1\}$ is a stationary time series.
\end{enumerate}
\end{prop}

Proposition \ref{modelgmc} implies that starting from any state $x$, the limiting distribution of the Markov chain $X_n(x)$ exists and the $n$-step transition probability measure $P^n(x,\cdot)$ converges weakly to $\pi$, as $n\rightarrow\infty$.

\subsection{Ergodicity}
In this section we further investigate the stability properties, including ergodicity and mixing for model (\ref{eq:expmodel}). Under the conditions of Proposition \ref{modelgmc}, the process $\{(X_t, Y_t)\}$ is strictly stationary, so we can extend it to be indexed by all the integers. The following proposition establishes ergodicity and absolute regularity when $Y_t$ is discrete.
\begin{prop}
\label{discreteergodicity}
Assume model (\ref{eq:expmodel}) where the support of $Y_t$ is a subset of $\mathbb{N}_0=\{0,1,\ldots,\}$, and that $g$ satisfies the contraction condition (\ref{ContractionFunction}). Then
\begin{enumerate}
\item There exists a measurable function $g_{\infty}:\mathbb{N}_0^{\infty}=\{(n_1, n_2, \ldots), n_i\in \mathbb{N}_0, i=1,2,\ldots\}\longrightarrow [0,\infty)$ such that $X_t=g_{\infty}(Y_{t-1}, Y_{t-2},\ldots)$ almost surely.
\item The count process $\{Y_t\}$ is absolutely regular with coefficients satisfying 
\begin{eqnarray*}
\beta(n)\le (a+b)^n/(1-(a+b)),
\end{eqnarray*}
and hence $\{(X_t, Y_t)\}$ is ergodic.
\end{enumerate}
\end{prop}

When $Y_t$ has a continuous distribution, geometric ergodicity of $\{X_t\}$ can be established under stronger conditions on $g$. The proof of the result relies on the classic Markov chain theory since $\{X_t\}$ is $\phi$-irreducible due to the continuity of the distribution in this situation.

\begin{prop}
\label{ContinuousErgodicity}
Assume model (\ref{eq:expmodel}) where the support of $Y_t$ is $[0, \infty)$, and that the function $g$ satisfies the contraction condition (\ref{ContractionFunction}). Moreover if $g$ is increasing and continuous in $(x, y)$, then
\begin{enumerate}
\item There exists $g_{\infty}:[0,\infty)^{\infty}\rightarrow [0,\infty)$ such that $X_t=g_{\infty}(Y_{t-1}, Y_{t-2},\ldots)$ almost surely.
\item The Markov chain $\{X_t, t\ge 1\}$ is geometrically ergodic provided that $a+b<1$, and hence $\{(X_t, Y_t)\}$ is stationary and ergodic.
\end{enumerate}
\end{prop}

\section{Likelihood Inference}
In this section, we consider maximum likelihood estimates of the parameters and study their asymptotic behavior, including consistency and asymptotic normality. Denote the $d-$dimensional parameter vector by $\theta\in \mathbb{R}^d$, i.e., $\theta=(\theta_1,\ldots, \theta_d)^T$, and the true parameter vector by $\theta_0=(\theta_1^0,\ldots,\theta_d^0)^T$. Then the likelihood function of model (\ref{eq:expmodel}) conditioned on $\eta_1$ and based on the observations $Y_1, \ldots, Y_n$ is given by
\begin{eqnarray*}
L(\theta|Y_1,\ldots,Y_n,\eta_1)=\displaystyle\prod_{t=1}^n \exp\{\eta_t(\theta) Y_t-A(\eta_t(\theta))\}h(Y_t),
\end{eqnarray*}
where $\eta_t(\theta)=B^{-1}(X_t(\theta))$ is updated through the iterations $X_t=g_{\theta}(X_{t-1}, Y_{t-1})$. The log-likelihood function, up to a constant independent of $\theta$, is given by
\begin{eqnarray}
l(\theta)=\displaystyle\sum_{t=1}^n l_t(\theta)=\sum_{t=1}^n \{\eta_t(\theta) Y_t-A(\eta_t(\theta))\},\label{eq:loglikeexp}
\end{eqnarray}
with score function
\begin{eqnarray}
S_n(\theta)=\frac{\partial l(\theta)}{\partial \theta}=\sum_{t=1}^n \{Y_t-B(\eta_t(\theta))\}\frac{\partial \eta_t(\theta)}{\partial \theta}.
\label{eq:scoreexp}
\end{eqnarray}
The maximum likelihood estimator $\hat{\theta}_n$ is a solution to the equation $S_n(\theta)=0$. Let $P_{\theta_0}$ be the probability measure under the true parameter $\theta_0$ and unless otherwise indicated, $\E[\cdot]$ is taken under $\theta_0$. Recall that $X_t=g_{\infty}^{\theta}(Y_{t-1}, Y_{t-2},\ldots)$ according to part (a) of Propositions \ref{discreteergodicity} and \ref{ContinuousErgodicity}. We will derive the asymptotic properties of the maximum likelihood estimator $\hat{\theta}_n$ based on a set of regularity conditions:
\begin{enumerate}
\item[(A0)] $\theta_0$ is an interior point in the compact parameter space $\Theta\in\mathbb{R}^d$.
\item[(A1)] For any $\theta\in \Theta$, $g^{\theta}_{\infty}\ge x_{\theta}^{\ast}\in \mathcal{R}(B)$, where $\mathcal{R}(B)$ is the range of $B(\eta)$. Moreover $x_{\theta}^{\ast}\ge x^{\ast}\in \mathcal{R}(B)$ for all $\theta$.
\item[(A2)] For any $\mathbf{y}\in [0,\infty)^{\infty}$ or $\mathbb{N}_0^{\infty}$, the mapping $\theta\mapsto g_{\infty}^{\theta}(\yy)$ is continuous.
\item[(A3)] $g(x,y)$ is increasing in $(x, y)$ if $Y_t$ given $\mathcal{F}_{t-1}$ has a continuous distribution.
\item[(A4)] $\E\{Y_1\sup_{\theta\in \Theta}B^{-1}(g_{\infty}^{\theta}(Y_0,Y_{-1},\ldots))\}<\infty$.
\item[(A5)] If there exists a $t\ge 1$ such that $X_t(\theta)=X_t(\theta_0)$, $P_{\theta_0}$-a.s., then $\theta=\theta_0$.
\item[(A6)] The mapping $\theta\mapsto g_{\infty}^{\theta}$ is twice continuously differentiable.
\item[(A7)] $\E\{B'(\eta_1(\theta_0))(\partial \eta_1(\theta)/\partial \theta_i)^2|_{\theta=\theta_0}\}<\infty$, for $i=1,\ldots,d$.
\end{enumerate}
Strong consistency of the estimates is derived according to the lemma below, which is adapted from Lemma 3.11 in \cite{Pfanzagl69}.
\begin{lemma}
\label{WaldConsistency}
Assume that $\Theta\subset \mathbb{R}^d$ is a compact set, and that $(\Omega, \mathcal{F}, P)$ is a probability space. Let $\{f_{\theta}: \mathbb{R}^{\infty}\mapsto [-\infty,\infty], \theta\in \Theta\}$ be a family of Borel measurable functions such that:
\begin{enumerate}
\item $\theta\mapsto f_{\theta}(\xx)$ is upper-semicontinuous for all $\mathbf{\xx}\in \mathbb{R}^{\infty}$.
\item $\sup_{\theta\in C}f_{\theta}(\xx)$ is Borel measurable for any compact set $C\subset \Theta$.
\item $\E\{\sup_{\theta\in\Theta}f_{\theta}(X)\}<\infty$ for some random variable $X$ defined on $(\Omega, \mathcal{F}, P)$.
\end{enumerate}
Then
\begin{enumerate}
\item $\theta\mapsto \E[f_{\theta}(X)]$ is upper-semicontinuous.
\item If $\{X_t: \Omega\mapsto \mathbb{R}^{\infty}, t\in \mathbb{Z}\}$ is an ergodic stationary process defined on $(\Omega, \mathcal{F}, P)$, and for all $t$, $X_t$ has the same distribution as $X$, then 
\begin{eqnarray*}
\limsup_{n\rightarrow\infty}\sup_{\theta\in C}\frac{1}{n}\sum_{i=1}^n f_{\theta}(X_i)\le \sup_{\theta\in C}\E\{f_{\theta}(X_1)\},~~\mbox{a.s.-}P,
\end{eqnarray*} 
for any compact set $C$.
\end{enumerate}
\end{lemma}
\cite{Pfanzagl69} proved the result assuming the independent structure of $\{X_t\}$, but the same result proves to be true provided that the strong law of large numbers can be applied. By virtue of Lemma \ref{WaldConsistency}, we can derive the strong consistency of the estimates.

\begin{thm}
\label{Consistency}
Assume model (\ref{eq:expmodel}) with the function $g$ satisfying the contraction condition (\ref{ContractionFunction}), and that assumptions (A0)-(A5) hold. Then the maximum likelihood estimator $\hat{\theta}_n$ is strongly consistent, that is, 
\begin{eqnarray*}
\hat{\theta}_n\stackrel{a.s.}{\longrightarrow}\theta_0,~~ \mbox{as}~n\rightarrow\infty.
\end{eqnarray*}
\end{thm}

The following theorem addresses the asymptotic distribution of the MLE and the idea of proof is similar to that in \cite{Davis03}. Unless otherwise indicated, $\eta_t$ and $\dot{\eta}_t$ are both evaluated at $\theta_0$, i.e., $\eta_t=\eta_t(\theta_0)$ and $\dot{\eta_t}=(\partial \eta_t/\partial \theta)|_{\theta=\theta_0}$.

\begin{thm}
\label{AsympNormal}
Assume model (\ref{eq:expmodel}) with the function $g$ satisfying the contraction condition (\ref{ContractionFunction}), and that assumptions (A0)-(A7) hold. Then the maximum likelihood estimator $\hat{\theta}_n$ is asymptotically normal, i.e.,
\begin{eqnarray*}
\sqrt{n}(\hat{\theta}_n-\theta_0)\stackrel{\mathcal{L}}{\longrightarrow}N(0, \Omega^{-1}),~~~ \mbox{as}~~n\rightarrow\infty,
\end{eqnarray*}
where $\Omega=\E\{B'(\eta_t)\dot{\eta}_t\dot{\eta}_t^T\}$. \\
\end{thm}

We remark that in practice, the population quantities in $\Omega$ can be replaced by their estimated counterparts. Examples of such substitution will be illustrated below in specific models.

\section{Examples}
\subsection{Linear dynamic models}
The conditional mean process $\{X_t\}$ in these models has GARCH-like dynamics. Specifically they are described as
\begin{eqnarray}
Y_t|\mathcal{F}_{t-1}\sim p(y|\eta_t),~~~X_t=\delta+\alpha X_{t-1}+\beta Y_{t-1}, \label{eq:LinearModel}
\end{eqnarray}
where $X_t=B(\eta_t)=\E(Y_t|\mathcal{F}_{t-1})$, and $\delta>0, \alpha, \beta\ge 0$ are parameters. Observe that model (\ref{eq:LinearModel}) is a special case of model (\ref{eq:expmodel}) by defining the function $g_{\theta}$ as
\begin{eqnarray}
g_{\theta}(x,y)=\delta+\alpha x+\beta y, \label{eq:LinearG}
\end{eqnarray}
with $\theta=(\delta,\alpha,\beta)^T$ and the contraction condition (\ref{ContractionFunction}) corresponds to $\alpha+\beta<1$. Note that by recursion we have, for all $t$,
\begin{eqnarray}
X_t(\theta)=\delta/(1-\alpha)+\beta\displaystyle\sum_{k=0}^{\infty}\alpha^k Y_{t-1-k}. \label{eq:InfinitePastRep}
\end{eqnarray}
It follows that $X_t(\theta)\ge x^{\ast}=\delta/(1-\alpha)$ since $Y_t$ only takes non-negative values. A direct application of Propositions \ref{modelgmc}, \ref{discreteergodicity} and \ref{ContinuousErgodicity} gives the stability properties of model (\ref{eq:LinearModel}).
\begin{prop}
\label{LinearStability}
Assume model (\ref{eq:LinearModel}) with $\alpha+\beta<1$. Then the process $\{X_t, t\ge 1\}$ has a unique stationary distribution $\pi$, and $\{(X_t, Y_t), t\ge 1\}$ is ergodic if $X_1\sim \pi$. 
\end{prop}
\medskip

If $\theta_0=(\delta_0, \alpha_0, \beta_0)^T$ denotes the true parameter vector, then the log-likelihood function $l(\theta)$ and the score function $S_n(\theta)$ of model (\ref{eq:LinearModel}) are given by (\ref{eq:loglikeexp}) and (\ref{eq:scoreexp}) respectively, where $\partial \eta_t(\theta)/\partial \theta=(\partial \eta_t/\partial \delta, \partial \eta_t/\partial \alpha, \partial \eta_t/\partial \beta)^T$ is determined recursively by
\begin{eqnarray}
\frac{\partial \eta_t}{\partial \theta}=\begin{pmatrix}
                                                   1 \\
                                                   B(\eta_{t-1})\\
                                                   Y_{t-1}
                                                  \end{pmatrix}/B'(\eta_t)+\alpha\frac{B'(\eta_{t-1})}{B'(\eta_t)}\frac{\partial \eta_{t-1}}{\partial \theta}.
\end{eqnarray}
The maximum likelihood estimator $\hat{\theta}_n$ is a solution of the equation $S_n(\theta)=0$. Furthermore, the Hessian matrix can be found by taking derivatives of the score function, i.e.,
\begin{eqnarray*}
H_n(\theta)=\frac{\partial^2 l(\theta)}{\partial \theta\partial \theta^T}=\sum_{t=1}^n [-B'(\eta_t(\theta))\frac{\partial\eta_t(\theta)}{\partial\theta}\frac{\partial \eta_t(\theta)}{\partial\theta^T}+\{Y_t-B(\eta_t(\theta))\}\frac{\partial^2\eta_t(\theta)}{\partial\theta\partial \theta^T}],
\end{eqnarray*}
where 
\small
\begin{eqnarray*}
\frac{\partial^2\eta_t}{\partial\theta\partial\theta^T}&=&\biggr(\frac{B''(\eta_t)}{(B'(\eta_t))^2}\frac{\partial\eta_t}{\partial \theta}~~~\frac{B'(\eta_{t-1})B'(\eta_t)}{(B'(\eta_t))^2}\frac{\partial\eta_{t-1}}{\partial \theta}-\frac{B'(\eta_{t-1})B''(\eta_t)}{(B'(\eta_t))^2}\frac{\partial\eta_{t}}{\partial \theta}\\
&&\frac{-Y_{t-1}B''(\eta_t)}{(B'(\eta_t))^2}\frac{\partial\eta_t}{\partial\theta}\biggr)+(0~~~1~~~0)^T\frac{B'(\eta_{t-1})}{B'(\eta_t)}\frac{\partial\eta_{t-1}}{\partial\theta^T}+\alpha\frac{B''(\eta_{t-1})B'(\eta_t)}{(B'(\eta_t))^2}\\
&&\frac{\partial\eta_{t-1}}{\partial\theta}\frac{\partial \eta_{t-1}}{\partial\theta^T}-\alpha\frac{B'(\eta_{t-1})B''(\eta_t)}{(B'(\eta_t))^2}\frac{\partial\eta_t}{\partial\theta}\frac{\partial \eta_t}{\partial\theta^T}+\alpha\frac{B'(\eta_{t-1})}{B'(\eta_t)}\frac{\partial^2\eta_{t-1}}{\partial\theta \partial\theta^T}. 
\end{eqnarray*}
\normalsize
It follows from the representation with the infinite past (\ref{eq:InfinitePastRep}) that assumptions (A1)-(A3) and (A6) are satisfied. In order to apply Theorem \ref{AsympNormal} when investigating the asymptotic behavior of the MLE, we need to impose the following regularity conditions:
\begin{enumerate}
\item[(L0)] The true parameter vector $\theta_0$ lies in a compact neighborhood $\Theta\in \mathbb{R}_+^3$ of $\theta_0$, where $\Theta=\{\theta=(\delta, \alpha, \beta)^T\in \mathbb{R}_+^3: 0<\delta_L\le \delta\le \delta_U, \epsilon\le \alpha+\beta\le 1-\epsilon\}$ for some $\epsilon>0$.
\item[(L1)] $\E\{Y_1\sup_{\theta\in \Theta}B^{-1}(\delta/(1-\alpha)+\beta\sum_{k=0}^{\infty}\alpha^k Y_{-k})\}<\infty$.
\item[(L2)] $\E\{B'(\eta_1(\theta_0))(\partial \eta_1(\theta)/\partial \theta_i)^2|_{\theta=\theta_0}\}<\infty$, for $i=1,2,3$.
\end{enumerate}
\begin{thm}
\label{LinearAsymp}
Assume model (\ref{eq:LinearModel}) and that assumptions (L0)-(L2) hold. Then the maximum likelihood estimator $\hat{\theta}_n$ is strongly consistent and asymptotically normal, i.e.,
\begin{eqnarray*}
\sqrt{n}(\hat{\theta}_n-\theta_0)\stackrel{\mathcal{L}}{\longrightarrow}N(0, \Omega^{-1}),~~~ \mbox{as}~~n\rightarrow\infty,
\end{eqnarray*}
where $\Omega=\E\{B'(\eta_t)\dot{\eta}_t\dot{\eta}_t^T\}$, where $\eta_t=\eta_t(\theta_0)$ and $\dot{\eta_t}=\frac{\partial \eta_t}{\partial \theta}|_{\theta=\theta_0}$.
\end{thm}
\medskip

\begin{remark}
\label{ARMARemark}
Under the contraction condition $\alpha+\beta<1$, $\{Y_t\}$ can be represented as a causal ARMA(1,1) process. To see this, denote $d_t=Y_t-X_t$, then it follows from $\E(d_t|\mathcal{F}_{t-1})=0$ that $\{d_t, t\in \mathbb{Z}\}$ is a martingale difference sequence. Therefore model (\ref{eq:LinearModel}) can be written as 
\begin{eqnarray}
Y_t-(\alpha+\beta)Y_{t-1}=\delta+d_t-\alpha d_{t-1}. \label{eq:ARMArepresentation}
\end{eqnarray}
Denote $\gamma_Y(h)$ as the auto-covariance function of $\{Y_t\}$. If $\gamma_Y(0)<\infty$, then $\gamma_Y(h)=(\alpha+\beta)^{h-1}\gamma_Y(1)$, for $h\ge 1$, see for example \cite{TSTM}.
\end{remark}

In practice, it can be difficult to verify assumptions (L1) and (L2), so we provide some alternative sufficient conditions for them in the following two remarks. 

\begin{remark}
\label{L1Remark}
A sufficient condition for assumption (L1) is 
\begin{eqnarray*}
\E\{Y_1B^{-1}(\delta_U/\epsilon+\displaystyle\sum_{k=1}^{\infty}(1-\epsilon)^k Y_{1-k})\}<\infty,
\end{eqnarray*}
provided that $\delta_U/\epsilon+\sum_{k=1}^{\infty}(1-\epsilon)^k Y_{1-k}$ is in the range of $B(\eta)$. This can be seen by noting that $X_1(\theta)\le \delta_U/\epsilon+\sum_{k=1}^{\infty}(1-\epsilon)^k Y_{1-k}$.
\end{remark}

\begin{remark}
\label{L2Remark}
If $A''(\eta_t)\ge \underline{c}$ for some $\underline{c}>0$, this is true, for example, when $A''(\eta)$ is increasing and $A''(B^{-1}(\delta_L))>0$, then a sufficient condition for assumption (L2) is $\gamma_Y(0)<\infty$.
\end{remark}
\medskip

Next we consider some specific models belonging to class (\ref{eq:LinearModel}), most of which are geared towards modeling time series of counts.

\begin{example}
\label{PoissonIngarchExample}
\noindent As a special case of the linear dynamic model (\ref{eq:LinearModel}) with $\eta_t=\log\lambda_t$ and $A(\eta_t)=e^{\eta_t}$, the Poisson INGARCH$(1, 1)$ model is given by
\begin{eqnarray}
Y_t|\mathcal{F}_{t-1} \sim \mbox{Pois}(\lambda_t),~~\lambda_t=\delta+\alpha\lambda_{t-1}+\beta Y_{t-1}, \label{eq:poisingarch11}
\end{eqnarray}
where $\delta>0, \alpha, \beta\ge 0$ are parameters. According to Proposition \ref{LinearStability}, it is easy to see that if $\alpha+\beta<1$, then $\{\lambda_t\}$ is geometric moment contracting and has a unique stationary distribution $\pi$; moreover if $\lambda_1\sim \pi$, then $\{(Y_t, \lambda_t), t\ge 1\}$ is an ergodic stationary process. As for inference, the MLE $\hat{\theta}_n$ is strongly consistent and asymptotically normal according to Theorem \ref{LinearAsymp}, i.e., $\sqrt{n}(\hat{\theta}_n-\theta_0)\stackrel{\mathcal{L}}{\longrightarrow}N(0, \Omega^{-1})$, as $n\rightarrow\infty$, where $\Omega=\E\{1/\lambda_t(\partial\lambda_t/\partial\theta)(\partial\lambda_t/\partial\theta)^T\}$. To see this, we only need to verify assumptions (L1) and (L2). Note that by \cite{Fokianos}, we have $\gamma_Y(0)=\{1-(\alpha+\beta)^2+\beta^2\}/\{1-(\alpha+\beta)^2\}$ and $\gamma_Y(h)=\mu C(\theta)(\alpha+\beta)^{h-1}$ for $h\ge 1$, where $\mu=\E Y_t=\delta/(1-\alpha-\beta)$ and $C(\theta)$ is a positive constant dependent on $\theta$. Hence by monotone convergence theorem, we have
\begin{eqnarray*}
\E[Y_1\log\{\delta_U/\epsilon+\displaystyle\sum_{k=1}^{\infty}(1-\epsilon)^k Y_{1-k}\}]&\le&\E[Y_1\{\delta_U/\epsilon+\displaystyle\sum_{k=1}^{\infty}(1-\epsilon)^k Y_{1-k}\}]\\
                  &=&\frac{\delta_U}{\epsilon} \E Y_1+\displaystyle\sum_{k=1}^{\infty}(1-\epsilon)^k\E Y_1Y_{1-k}\\
                  &=&\mu\frac{\delta_U}{\epsilon}+\displaystyle\sum_{k=1}^{\infty}(1-\epsilon)^k\{\gamma_Y(k)+\mu^2\}<\infty.
\end{eqnarray*}
Hence assumption (L1) holds according to Remark \ref{L1Remark}. Notice that $B(\eta_t)=\lambda_t\ge \lambda^{\ast}:=\delta/(1-\alpha)$ for all $t$, so $A''(\eta_t)=e^{\eta_t}$ is bounded away from 0, so assumption (L2) holds according to Remark \ref{L2Remark}.
\end{example}
\medskip

Moreover, the iterated random function approach can be used to study the properties of INGARCH models with higher orders. A Poisson INGARCH($p,q$) model takes the form
\begin{eqnarray}
Y_t|\mathcal{F}_{t-1}\sim \mbox{Pois}(\lambda_t),~~\lambda_t=\delta+\displaystyle\sum_{i=1}^p \alpha_i\lambda_{t-i}+\sum_{j=1}^q \beta_j Y_{t-j}, \label{eq:PoisIngarchpq}
\end{eqnarray}
where $\delta>0, \alpha_i, \beta_j\ge 0, i=1,\ldots, p$; $j=1,\ldots, q$. Applying similar ideas as in the INGARCH($1,1$) case, we have the following stationarity result.

\begin{prop}
\label{poissonpq}
Consider the INGARCH$(p,q)$ model (\ref{eq:PoisIngarchpq}) and suppose $\sum_{i=1}^p \alpha_i + \sum_{j=1}^q \beta_j<1$, then $\{\lambda_t\}$ is geometric moment contracting and has a unique stationary distribution.
\end{prop}
\medskip

\begin{example}
\label{NbIngarchExample}
The negative binomial INGARCH$(1, 1)$ model (NB-INGARCH) is defined as
\begin{eqnarray}
Y_t|\mathcal{F}_{t-1}\sim \mbox{NB}(r,p_t), ~~X_t=\delta+\alpha X_{t-1}+\beta Y_{t-1}, \label{eq:nbingarch11}
\end{eqnarray}
where $X_t=r(1-p_t)/p_t$, $\delta>0,\alpha,\beta\ge 0$ are parameters and the notation $Y\sim\mbox{NB}(r, p)$ represents the negative binomial distribution with probability mass function given by
\begin{eqnarray*}
P(Y=k)={k+r-1 \choose r-1}(1-p)^k p^r, ~~~~~~ k=0,1,2,\ldots.
\end{eqnarray*}
When $r=1$, the conditional distribution of $Y_t$ becomes geometric distribution with probability of success $p_t$, in which case (\ref{eq:nbingarch11}) reduces to a geometric INGARCH model.

By virtue of Proposition \ref{LinearStability}, if $\alpha+\beta<1$, then $\{X_t, t\ge 1\}$ is a geometric moment contracting Markov chain, and has a unique stationary distribution $\pi$; and when $X_1\sim \pi$, $\{(X_t, Y_t), t\ge 1\}$ is ergodic. As for inference, we can first estimate $\theta=(\delta, \alpha, \beta)^T$ for $r$ fixed and calculate the profile likelihood as a function of $r$. Then $r$ is estimated by choosing the one which maximizes the profile likelihood, and thus $\hat{\theta}$ can be otained correspondingly. Moreover, if we assume $r$ is known and $(\alpha+\beta)^2+\beta^2/r<1$, then under assumption (L0), the maximum likelihood estimator $\hat{\theta}_n$ is strongly consistent and asymptotically normal with mean $\theta_0$ and covariance matrix $\Omega^{-1}/n$, where $\Omega=\E\{r/X_t/(X_t+r)(\partial X_t/\partial \theta)(\partial X_t/\partial\theta)^T\}$. Verification of assumptions (L1) and (L2) is sufficient to demonstrate the result. Since $B^{-1}(x)=\log\{x/(x+r)\}<0$, so assumption (L1) holds according to Remark \ref{L1Remark}. Note that $A''(\eta_t)=re^{\eta_t}/(1-e^{\eta_t})^2$ is increasing, so assumption (L2) holds provided $\gamma_Y(0)<\infty$ according to Remark \ref{L2Remark}. Because $\var(X_1)=\alpha^2\var(X_0)+\beta^2\var(Y_0)+2\alpha\beta\cov(X_0, Y_0)$, where 
\begin{eqnarray*}
\var(Y_0)&=&\E\{\var(Y_0|X_0)\}+\var\{\E(Y_0|X_0)\}\\
             &=&\E\{r(1-p_0)/p_0^2\}+\var(X_0)=\mu+1/r\E X_0^2+\var(X_0),
\end{eqnarray*}
and $\cov(X_1, Y_1)=\E Y_1X_1-\mu^2=\E X_1^2-\mu^2=\var(X_1)$, it follows from the stationarity that 
\begin{eqnarray*}
\var(X_0)=\frac{\beta^2\mu(1+\mu/r)}{1-(\alpha+\beta)^2-\beta^2/r}.
\end{eqnarray*}
Hence $\gamma_Y(0)<\infty$ provided $(\alpha+\beta)^2+\beta^2/r<1$.
\end{example}
\medskip

\begin{example}
We define the binomial INGARCH$(1, 1)$ model as
\begin{eqnarray}
Y_t|\mathcal{F}_{t-1}\sim \mbox{B}(m, p_t),~~mp_t=\delta+\alpha mp_{t-1}+\beta Y_{t-1}, \label{eq:BinomIngarch11}
\end{eqnarray}
where $\delta>0, \alpha, \beta\ge 0$ are parameters and $\delta+\alpha m+\beta m\le m$ since $p_t\in (0, 1)$. This implies the contraction condition $\alpha+\beta<1$. In particular, when $m=1$, it models time series of binary data, and is called a Bernoulli INGARCH model. If $\delta+\alpha m+\beta m\le m$, then $\{X_t=mp_t, t\ge 1\}$ is geometric moment contracting and has a unique stationary distribution $\pi$; furthermore, $\{(X_t, Y_t), t\ge 1\}$ is ergodic when $X_1\sim \pi$. 

We now consider the inference of the model. Firstly, because of the special constraint $p_t\in (0, 1)$, the parameter space becomes 
\begin{eqnarray*}
\Theta=\{(\delta,\alpha,\beta)^T: 0<\delta_L\le \delta\le \delta_U, \epsilon\le \alpha+\beta\le 1-\epsilon\}~~\mbox{for some}~~\epsilon>\delta_U/m.
\end{eqnarray*}
Since $Y_t\le m$, so $X_1(\theta)\le (\delta+\alpha m)/(1-\alpha)$ and $B^{-1}(X_1(\theta))\le \log\{(\delta_U+(1-\epsilon)m)/(\epsilon m-\delta_U)\}$. Hence assumption (L1) holds. Notice that $A''(\eta_t)=mp_t(1-p_t)$ and $p_t\in[\delta_U/m, (\delta+\beta m)/(m(1-\alpha))]\subsetneq [0, 1]$, so $A''(\eta_t)$ is bounded away from 0. Similar to the proof in Example \ref{NbIngarchExample}, one can show that $\gamma_Y(0)<\infty$ provided that $(\alpha+\beta)^2+\beta^2/m<1$. So assuming $m$ is known and $(\alpha+\beta)^2+\beta^2/m<1$, the maximum likelihood estimator $\hat{\theta}_n$ is strongly consistent and asymptotically normal with mean $\theta_0$ and covariance matrix $\Omega^{-1}/n$, where $\Omega=\E\{m/X_t/(m-X_t)(\partial X_t/\partial\theta)(\partial X_t/\partial\theta)^T\}$. 
\end{example}

\begin{example}
The gamma INGARCH model, which has a continuous response, is given by
\begin{eqnarray}
Y_t|\mathcal{F}_{t-1}\sim \Gamma(\kappa, s_t),~~s_t=\delta/\kappa+\alpha s_{t-1}+\beta/\kappa Y_{t-1}, \label{eq:GammaIngarch11}
\end{eqnarray}
where $\kappa$ and $s_t$ are the shape and scale parameters of the gamma distribution respectively and $\delta>0,\alpha,\beta\ge 0$ are parameters. Here the natural parameter is $\eta_t=-1/s_t$ and the Markov chain $X_t=B(\eta_t)=-\kappa/\eta_t$. If $\alpha+\beta<1$, then $\{X_t=\kappa s_t, t\ge 1\}$ is geometric moment contracting and has a unique stationary distribution $\pi$; furthermore, $\{(Y_t, X_t), t\ge 1\}$ is an ergodic stationary process if $X_1\sim \pi$.

As for the inference in this model, assume $\kappa$ is known and $(\alpha+\beta)^2+\beta^2/\kappa<1$. Then the maximum likelihood estimator $\hat{\theta}_n$ is strongly consistent and asymptotically normal with mean $\theta_0$ and covariance matrix $\Omega^{-1}/n$ where $\Omega=\E\{\kappa/s_t^2(\partial s_t/\partial \theta)(\partial s_t/\partial\theta)^T\}$. To see this, note that $B^{-1}(x)=-\kappa/x<0$ when $x>0$, which verifies assumption (L1) according to Remark \ref{L1Remark}. Similar to the proof in Example \ref{NbIngarchExample}, one can show that $\gamma_Y(0)=(1/\kappa+1)\gamma_X(0)+\mu^2/\kappa$ and $\gamma_X(0)=(\beta^2\mu^2/\kappa)/\{1-(\alpha+\beta)^2-\beta^2/\kappa\}$. Hence as long as $(\alpha+\beta)^2+\beta^2/\kappa<1$, we have $\gamma_Y(0)<\infty$. Since $A''(\eta_t)=\kappa/\eta_t^2\ge \delta_L^2/\kappa>0$, assumption (L2) holds according to Remark \ref{L2Remark}.
\end{example}

\subsection{Nonlinear dynamic models}
It is possible to generalize (\ref{eq:LinearModel}) to nonlinear dynamic models. One approach is based on the idea of spline basis functions, see for example, \cite{SemiRegression}. In this framework, the model specification is given by
\begin{eqnarray}
Y_t|\mathcal{F}_{t-1}\sim p(y|\eta_t),~~X_t=\delta+\alpha X_{t-1}+\beta Y_{t-1}+\displaystyle\sum_{k=1}^K\beta_k(Y_{t-1}-\xi_k)^+,
\label{eq:NonLinear}
\end{eqnarray}
where $K\in \mathbb{N}_0$, $\delta>0, \alpha,\beta\ge 0, \beta_1,\ldots, \beta_K$ are parameters, $\{\xi_k\}_{k=1}^K$ are the so-called \emph{knots}, and $x^+$ is the positive part of $x$. In particular, when $K=0$, (\ref{eq:NonLinear}) reduces to the linear model (\ref{eq:LinearModel}). It is easy to see that model (\ref{eq:NonLinear}) is a special case of model (\ref{eq:expmodel}) by defining $g_{\theta}(x, y)=\delta+\alpha x+\beta y+\sum_{k=1}^K \beta_k(y-\xi_k)^+,$ where $\theta=(\delta, \alpha, \beta, \beta_1,\ldots, \beta_K)^T$. Note that in each of the pieces segmented by the knots, (\ref{eq:NonLinear}) has INGARCH-like dynamics. For example, if $Y_{t-1}\in [\xi_s, \xi_{s+1})$ for some $s< K$, then $X_t=(\delta-\sum_{k=1}^s \beta_k \xi_k) + \alpha X_{t-1}+ (\beta+\sum_{k=1}^s\beta_k) Y_{t-1}$. This can be viewed as one of the generalizations (e.g., \cite{ThreshGLM})  to the threshold autoregressive model (\cite{Tong90}). According to Propositions \ref{modelgmc}, \ref{discreteergodicity} and \ref{ContinuousErgodicity}, we can establish the stability properties of the model.

\begin{prop}
\label{NonLinearStability}
Consider model (\ref{eq:NonLinear}) with parameters satisfying $\alpha+\beta<1, \beta+\sum_{k=1}^s \beta_k\ge 0$ and $\alpha+\beta+\sum_{k=1}^s \beta_k<1$ for $s=1,\ldots, K$, then $\{X_t\}$ is geometric moment contracting and has a unique stationary distribution $\pi$. Moreover if $X_1\sim \pi$, then $\{(X_t, Y_t), t\ge 1\}$ is ergodic.
\end{prop}
\medskip

We now consider inference for this model. Assume the knots $\{\xi_k\}_{k=1}^K$ are known for $K$ fixed. Then the parameter vector $\theta=(\delta, \alpha, \beta, \beta_1,\ldots, \beta_K)^T$ can be estimated by maximizing the conditional log-likelihood function, which is available according to (\ref{eq:loglikeexp}). The number of knots $K$ can be selected by virtue of an information criteria, such as AIC and BIC. As for the locations of knots,  there are different strategies one can adopt for choosing them. One method is to place the knots at the $\{j / (K + 1), j = 1, \ldots, K\}$ quantiles of the population, which can be estimated from the data. A second method is to choose the locations that maximize the log likelihood. We will employ both procedures to real datasets in the next section. 

To study the asymptotic behavior of the estimates, first note that by iterating the recursion,
\begin{eqnarray}
X_t&=&\delta/(1-\alpha)+\beta\displaystyle\sum_{i=0}^{\infty}\alpha^i Y_{t-1-i}+\sum_{k=1}^K \beta_k\sum_{i=0}^{\infty}\alpha^i (Y_{t-1-i}-\xi_k)^+ \nonumber\\
      &=&\delta/(1-\alpha)+\displaystyle\sum_{i=0}^{\infty}\alpha^i\{\beta Y_{t-1-i}+\sum_{k=1}^K \beta_k(Y_{t-1-i}-\xi_k)^+\}.
\end{eqnarray}
This defines the function $g_{\infty}^{\theta}$ as in $X_t=g_{\infty}^{\theta}(Y_{t-1}, Y_{t-2},\ldots)$ and also verifies assumptions (A1)-(A3). Hence in order to apply Theorem \ref{LinearAsymp}, we only need to impose the following regularity assumptions for the nonlinear model (\ref{eq:NonLinear}):
\begin{enumerate}
\item[(NL1)] $\theta_0$ is an interior point in the parameter space $\Theta$, which is a compact subset of the parameter set satisfying the conditions in Proposition \ref{NonLinearStability}.
\item[(NL1)] $\E[Y_1\displaystyle\sup_{\theta\in \Theta}B^{-1}((\delta/(1-\alpha)+\sum_{i=0}^{\infty}\alpha^i\{\beta Y_{t-1-i}+\sum_{k=1}^K \beta_k(Y_{t-1-i}-\xi_k)^+\})]<\infty$.
\item[(NL2)] $\E[B'(\eta_1(\theta_0))\{\partial \eta_1(\theta)/\partial \theta_i)\}^2|_{\theta=\theta_0}]<\infty$, for $i=1,\ldots, K+3$.
\end{enumerate}
Sufficient conditions for assumptions (NL1) and (NL2) can be established similarly to those given in Remarks \ref{L1Remark} and \ref{L2Remark}. The asymptotic properties of the MLE are summarized in the following theorem.
\begin{thm}
\label{NonLinearAsympNormal}
For model (\ref{eq:NonLinear}), suppose that the placement of the knots is known, and that assumptions (NL0)-(NL2) hold, then the maximum likelihood estimator $\hat{\theta}_n$ is strongly consistent and asymptotically normal, i.e.,
\begin{eqnarray*}
\sqrt{n}(\hat{\theta}_n-\theta_0)\stackrel{\mathcal{L}}{\longrightarrow}N(0, \Omega^{-1}),~~\mbox{as}~~n\rightarrow\infty,
\end{eqnarray*}
where $\Omega=\E\{B'(\eta_t)\dot{\eta}_t\dot{\eta}_t^T\}$.
\end{thm}
\medskip

We use the Poisson nonlinear dynamic model as an illustrative example of the above results and refer readers to Section 5 for implementation of the estimation procedure. The model is defined as 
\begin{eqnarray}
Y_t|\mathcal{F}_{t-1}\sim \mbox{Pois}(\lambda_t),~~\lambda_t=\delta+\alpha \lambda_{t-1}+\beta Y_{t-1}+\displaystyle\sum_{k=1}^K\beta_k(Y_{t-1}-\xi_k)^+.
\label{eq:PoisNonLinear}
\end{eqnarray}
It follows that under the conditions of Proposition \ref{NonLinearStability} and Theorem \ref{NonLinearAsympNormal} that $\{(\lambda_t, Y_t), t\ge 1\}$ is a stationary and ergodic  process, and the estimates are strongly consistent and asymptotically normal. In practice the covariance matrix of the estimates can be obtained by recursively applying
\small
\begin{eqnarray*}
\frac{\partial\lambda_t}{\partial\theta}=\begin{pmatrix}
1 & \lambda_{t-1} & Y_{t-1} & (Y_{t-1}-\xi_1)^+ & \ldots & (Y_{t-1}-\xi_K)^+
\end{pmatrix}^T + \alpha\frac{\partial\lambda_{t-1}}{\partial\theta}.
\end{eqnarray*}
\normalsize
Another example of nonlinear dynamic models is the Poisson exponential autoregressive model proposed by \cite{Fokianos}, and it is given by
\begin{eqnarray}
\label{eq:PoisExpModel}
Y_t|\mathcal{F}_{t-1}\sim\mbox{Pois}(\lambda_t),~~\lambda_t=(\alpha_0+\alpha_1 \exp\{-\gamma \lambda_{t-1}^2\})\lambda_{t-1}+\beta Y_{t-1},
\end{eqnarray}
where $\alpha_0,\alpha_1,\beta, \gamma>0$ are parameters. We point out that if $\alpha_0+\alpha_1+\beta<1$, then model (\ref{eq:PoisExpModel}) belongs to the class of models (\ref{eq:expmodel}) and hence enjoys the stability properties stated in Propositions \ref{modelgmc} and \ref{discreteergodicity}. As for the inference of the model, we refer readers to \cite{Fokianos} for details.

\section{Numerical results}
The performance of the estimation procedure for the Poisson nonlinear dynamic model is illustrated in a simulation study. The MLE is obtained by optimizing the log-likelihood function (\ref{eq:loglikeexp}) using a Newton-Raphson method. Simulation results of the Poisson INGARCH can be found in \cite{Fokianos}. Other models including the negative binomial linear and nonlinear dynamic models and the exponential autoregressive model (\ref{eq:PoisExpModel}) will be applied to two real datasets, and tools for checking goodness of fit will be considered.

\subsection{Simulation for the nonlinear model}
As specified in (\ref{eq:PoisNonLinear}), a 1-knot nonlinear dynamic model is simulated according to
\begin{eqnarray*}
Y_t|\mathcal{F}_{t-1}\sim \mbox{Pois}(\lambda_t),~~\lambda_t=0.5+0.5\lambda_{t-1}+0.4Y_{t-1}-0.2(Y_{t-1}-5)^+
\end{eqnarray*}
with different sample sizes. Each sample size and parameter configuration is replicated $1000$ times. For each realization, the first $500$ simulated observations are discarded as burn-in in order to let the process reach its stationary regime. We first estimate the parameters assuming that the location of the knot is known, i.e., the true underlying model is (\ref{eq:NonLinear}) with only one knot at 5. The means and standard errors of the estimates from all 1000 runs are summarized in Table \ref{tab:simulation} and the histograms of the estimates are depicted in Figure \ref{fig:1_known_knot}. The performance of these estimates is reasonably good and consistent with the theory described in Theorem \ref{NonLinearAsympNormal}. As for estimating the parameters without knowing the location of the knots, the corresponding results of the MLE obtained by fitting a 1-knot model to all the 1000 replications are summarized in Table \ref{tab:1_unknown}. Here the locations of the knots are determined by sample quantiles. Not surprisingly, the performance of the maximum likelihood estimates of $\beta$ and $\beta_1$ is not as good as in the known knot case. However, the overall model performance, as reflected in the computation of the scoring rules (described in the next section), is competitive with the known knot case. For instance when $n=1000$, the means of ranked probability scores (RPS) for known and unknown knot cases are $1.0906$ and $1.0914$, respectively.
  \begin{table}
  \caption{\label{tab:simulation}Estimation results for 1-knot model with known knot location}
  \centering
  \resizebox{11cm}{!}{
\begin{tabular}{| c | c  c  c  c  c| }
\hline
           & $\delta$  & $\alpha$ & $\beta$  & $\beta_1$  & $n$ \\ \hline
True     & 0.5 & 0.5 & 0.4 & -0.2 &    \\ 
Estimates  & 0.5596 & 0.4861 & 0.3990  & -0.2009 & 500 \\
s.e.    & (0.0087) & (0.0030) & (0.0026) & (0.0051)  & \\ 
Estimates  & 0.5265 & 0.4944 & 0.3991  & -0.2016 & 1000  \\
s.e.    & (0.0041) & (0.0016) & (0.0013) & (0.0025)  &\\ \hline
\end{tabular}}
\end{table}

 \begin{table}
  \caption{\label{tab:1_unknown}Estimation for 1-knot model with unknown knot location}
  \centering
  \resizebox{11cm}{!}{
\begin{tabular}{| c | c  c  c  c  c | }
\hline
           & $\delta$  & $\alpha$ & $\beta$  & $\beta_1$  & $n$   \\ \hline
True     & 0.5 & 0.5 & 0.4 & -0.2 & \\
Estimates  & 0.5387 & 0.4852 & 0.4187  & -0.1614 & 500\\
s.e.    & (0.0089) & (0.0030) & (0.0031) & (0.0047)  & \\ 
Estimates  & 0.5002 & 0.4943 & 0.4197  & -0.1679 & 1000 \\
s.e.    & (0.0042) & (0.0016) & (0.0015) & (0.0023)  & \\ \hline
\end{tabular}}
\end{table}

\begin{figure}
\centering
\makebox{\includegraphics[scale=.6]{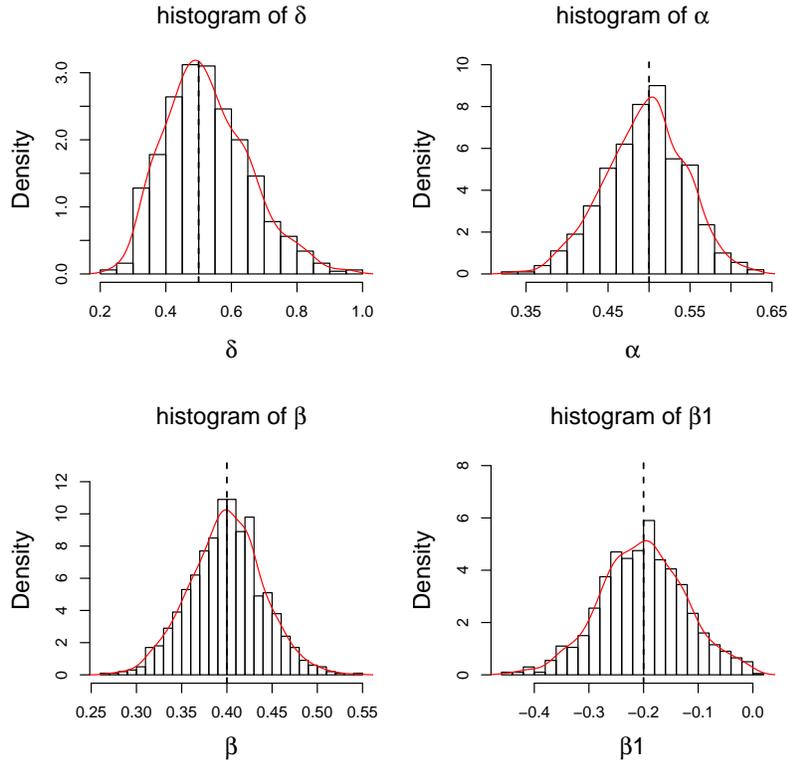}}
\caption{Histograms of the 1-knot model with sample size 1000 assuming the knot is known. The overlaying curves are the density estimates and the dashed vertical lines represent the true values of the parameters.}
\label{fig:1_known_knot}
\end{figure}

Next we turn to the problem of selecting the number of knots using an information criterion. Simulations with different sample sizes are implemented and the model selection results are summarized in Table \ref{tab:sim_1_unknown}. Numbers in the table stand for the proportion of times that each particular model is selected in the 1000 runs. For AIC, the 1-knot model is selected most often followed by a 2-knot model, at least in the cases when $n=1000$.
  \begin{table}
    \caption{\label{tab:sim_1_unknown}Model selection of 1-knot simulation}
  \centering
\resizebox{13.5cm}{!}{
\begin{tabular}{| c | c  c  c  c  c  c|}
\hline
Criteria                  &  0 knot &       1 knot     & 2 knots  &    3 knots  &     $\ge4$ knots & $n$\\ 
\hline
AIC               &  $34.3\%$ & $37.6\%$  &  $20.9\%$  & $5.2\%$  &  $2.0\%$ & 500\\ 
BIC              &  $80.5\%$  &  $18.8\%$   &  $0.6\%$  &  $0.1\%$  &  0 & \\ \hline
AIC               &  $12.4\%$ & $45.0\%$  &  $29.9\%$  & $8.3\%$  &  $4.4\%$  & 1000\\ 
BIC              & $59.4\%$  &  $38.4\%$   &  $2.0\%$  &  $0.2\%$   & 0 &  \\ \hline
\end{tabular}}
\end{table}
\normalsize

In light of the idea of interpolating the nonlinear dynamic of $\lambda_t$ by a piecewise linear function, we plot in Figure \ref{fig:curve_1_unknown} the fitted functions $\hat{\beta}y+\sum_{k=1}^K \hat{\beta}_k (y-\hat{\xi}_k)^+$ for each run of the simulations against its true form $0.4y-0.2(y-5)^+$. From the graph, we can see that the piecewise linear function fitted by the 1-knot model is closest to the true curve.
\begin{figure}
\centering
\makebox{\includegraphics[scale=.8]{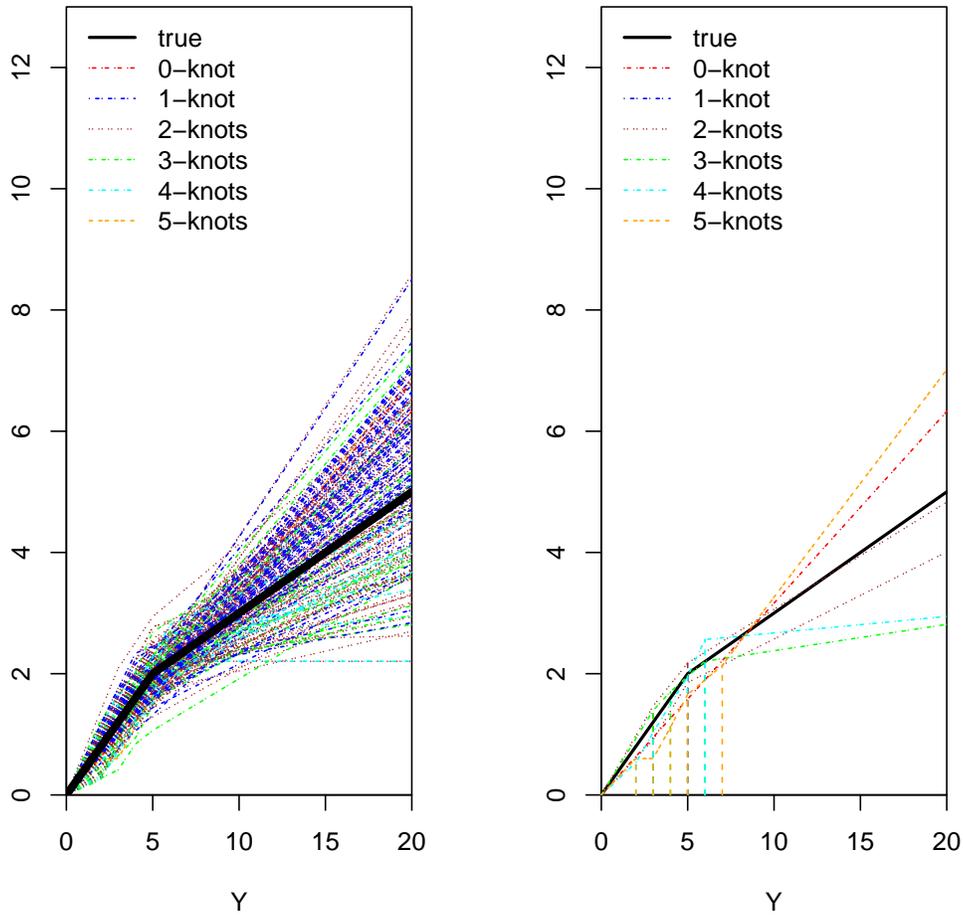}}
\caption{Left: the black curve is the true function $0.4y-0.2(y-5)^+$, and the other curves are the piecewise linear functions fitted in each simulation where the number of knots $K$ is selected via AIC; Right: for each value of $K$, we plot the fitted curve from one specific run that chooses the particular number of knots.}
\label{fig:curve_1_unknown}
\end{figure}

\subsection{Two data applications}

\subsubsection*{1. Number of transactions of Ericsson stock}

As an illustrative example, both linear and nonlinear dynamic models are employed to fit the number of transactions per minute for the stock Ericsson B during July 2nd, 2002 which consists of 460 observations. Figure \ref{fig:Ericsson_data} plots the data and the autocorrelation function. The positive dependence displayed in the data suggests the application of the models in our study.
\begin{figure}
\centering
\makebox{\includegraphics[scale = .48]{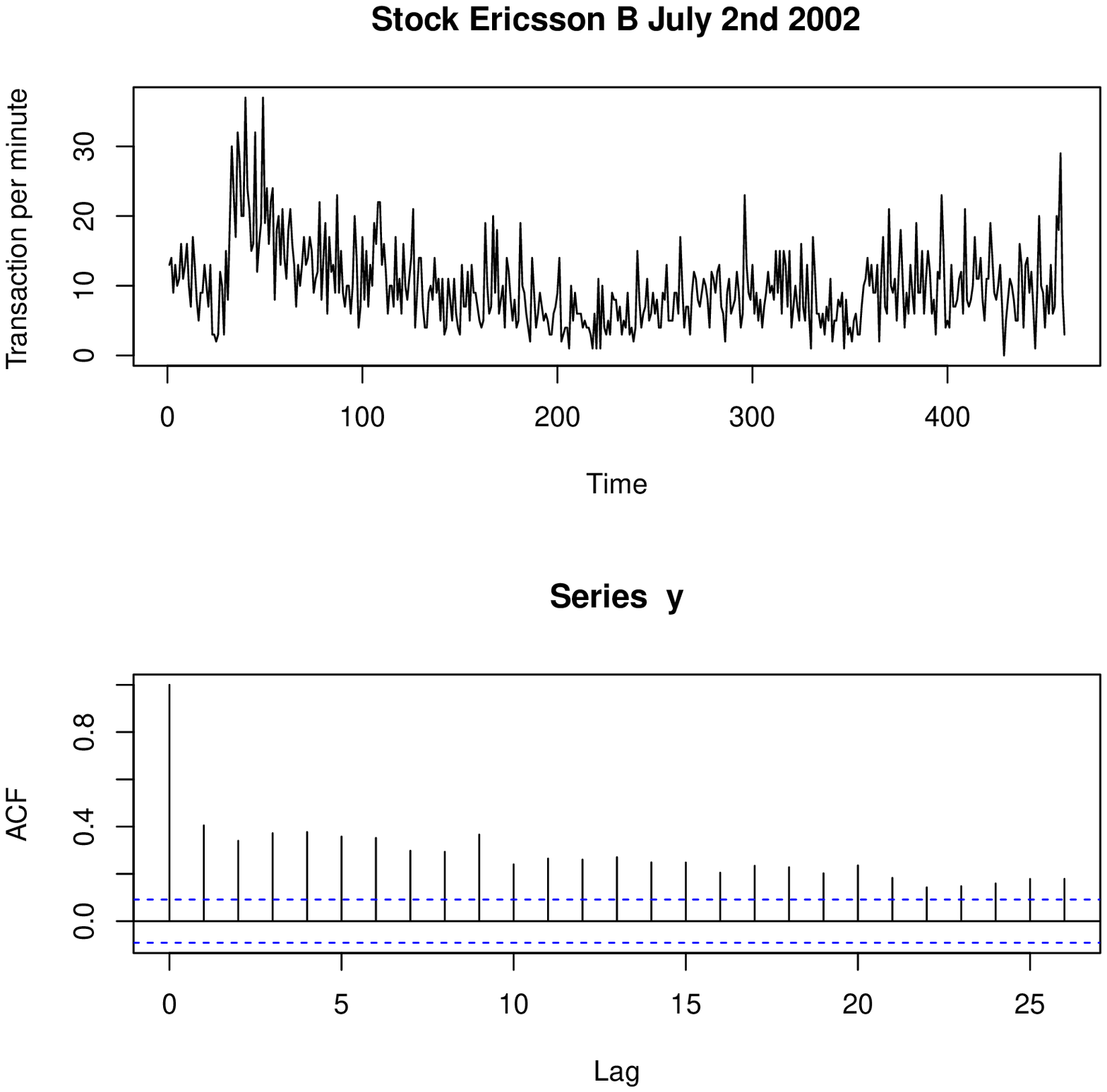}}
\caption{Top: Number of transactions per minute of the stock Ericsson B during July 2nd 2002; Bottom: ACF of the data.}
\label{fig:Ericsson_data}
\end{figure}

By computing the MLE of the parameters, the fitted Poisson INGARCH model is given by
\begin{eqnarray*}
\hat{\lambda}_t&=&0.2912+0.8312\hat{\lambda}_{t-1}+0.1395Y_{t-1},\\
&&(0.1000)~(0.0242)~~~~~~~(0.0188)
\end{eqnarray*}
and the fitted NB-INGARCH model is
\begin{eqnarray*}
Y_t|\mathcal{F}_{t-1}\sim \mbox{NB}(8, \hat{p}_t),~~\hat{X}_t&=& 0.2676+ 0.8447\hat{X}_{t-1}+0.1282Y_{t-1},\\
&&(0.1406)~(0.0350)~~~~~~~(0.0274)
\end{eqnarray*}
where $\hat{X}_t=8(1-\hat{p}_t)/\hat{p}_t$. The standard deviations in the parentheses are calculated according to the remark after Theorem \ref{AsympNormal}. 

As for the Poisson nonlinear dynamic model, AIC and BIC are used to help select the number of knots among 0 to 5; the values are reported in Table \ref{tab:infor_criteria}. 
  \begin{table}
    \caption{\label{tab:infor_criteria}Model selection results for Ericsson data}
  \centering
  \resizebox{13.5cm}{!}{
\begin{tabular}{| l |  c  c  c  c  c c |}
\hline
  & 0-knot & 1-knot & 2-knot & 3-knot & 4-knot & 5-knot \\ \hline 
LogL & -1433.19 &  -1431.21 & -1431.08& -1430.58& $\boldsymbol{-1429.65}$ & -1431.12 \\ 
AIC  & 2874.38& $\boldsymbol{2872.41}$ & 2874.17 & 2875.17 & 2875.30 & 2880.25 \\ 
BIC & $\boldsymbol{2890.90}$& 2893.07 & 2898.95& 2904.08 &  2908.35 & 2917.43  \\ \hline 
\end{tabular}}
\end{table}
The fitted 1-knot Poisson model, which has the smallest AIC, is given by
\begin{eqnarray*}
\hat{\lambda}_t&=&0.5837+0.8319\hat{\lambda}_{t-1}+0.0906Y_{t-1}+0.0722(Y_{t-1}-9)^+.\\
&&(0.1884)~(0.0241)~~~~~~~(0.0295)~~~~~~~(0.0373)
\end{eqnarray*} 
Note that the AIC values of the 2-knot and 3-knot models are both close to that of the 1-knot model, and therefore are used as a basis for comparison with the minimum AIC model. These models are given by $\hat{\lambda}_t=0.5519+0.8326\hat{\lambda}_{t-1}+0.0961Y_{t-1}+0.0154(Y_{t-1}-7)^++0.0559(Y_{t-1}-11)^+$ and $\hat{\lambda}_t=0.3614+0.8361\hat{\lambda}_{t-1}+0.1206Y_{t-1}+0.0433(Y_{t-1}-6)^+-0.0914(Y_{t-1}-9)^++0.0914(Y_{t-1}-13)^+$, respectively. 

As can be seen from the model checking below, the negative binomial INGARCH model seems to outperform the Poisson-based models. This could be explained by the over-dispersion exhibited by the data, since the mean and variance are 9.91 and 32.84, respectively. To this end, we fit the nonlinear negative binomial models and select the number of knots by minimizing the AIC. It turns out that the AIC value of a 1-knot model is the second smallest among all the candidates, with 2674.69 compared to the smallest value 2674.04, which is attained by the negative binomial INGARCH model fitted above. The fitted 1-knot negative binomial nonlinear model is given by $Y_t|\mathcal{F}_{t-1}\sim \mbox{NB}(8, \hat{p}_t)$, where $\hat{X}_t=8(1-\hat{p}_t)/\hat{p}_t$ follows
\begin{eqnarray*}
\hat{X}_t&=&0.4931+0.8444\hat{X}_{t-1}+0.0903Y_{t-1}+0.0603(Y_{t-1}-9)^+.\\
&&(0.2559)~(0.0350)~~~~~~~(0.0412)~~~~~~~(0.0546)
\end{eqnarray*}
Here the locations of knots for the nonlinear dynamic model are all estimated by the corresponding sample quantiles. We also tried estimating the knots by maximizing the likelihood, and in this application, the results by both methods are nearly identical. The exponential autoregressive model (\ref{eq:PoisExpModel}) is also applied to this dataset by \cite{Fokianos} and is given by
\begin{eqnarray*}
\hat{\lambda}_t&=&(0.8303+7.030\exp\{-0.1675\hat{\lambda}_{t-1}^2\})\hat{\lambda}_{t-1}+0.1551Y_{t-1}.\\
&&(0.0232)~(3.0732)~~~~~~(0.0592)~~~~~~~~~~~~~~~(0.0218)
\end{eqnarray*}

To assess the adequacy of the fit by all of the above models, we will consider an array of graphical and quantitative diagnostic tools for time series, some of which are specifically designed for time series of counts. Readers can refer to \cite{Davis03} and \cite{Jung11} for a comprehensive treatment of the tools. In our study, we first consider the standardized Pearson residuals $e_t=(Y_t-\E(Y_t|\mathcal{F}_{t-1}))/\sqrt{\var(Y_t|\mathcal{F}_{t-1})}$ which can be obtained by replacing the population quantities by their estimated counterparts. If the model is correctly specified, then the residuals $\{\hat{e}_t\}$ should be a white noise sequence with constant variance. It turns out that all the models considered above give very similar fitted conditional mean processes and the standardized Pearson residuals appear to be white. Figure \ref{fig:ericsson_fit} displays the fitted result for the 1-knot negative binomial model.
\begin{figure}
\centering
\makebox{\includegraphics[scale=.45]{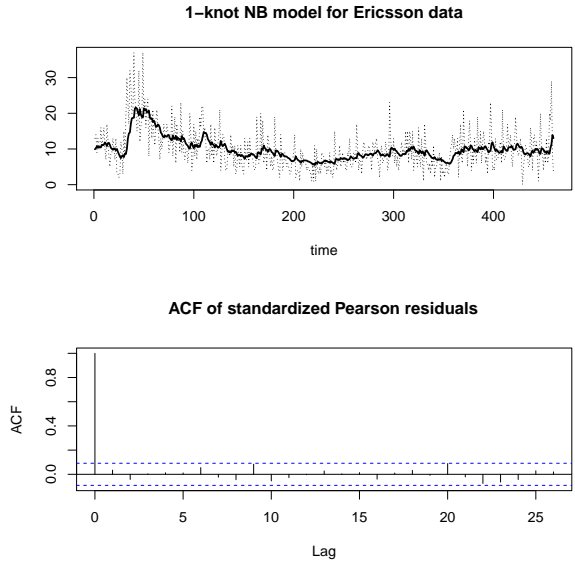}}
\caption{Top: Dotted curve represents the number of transactions of Ericsson stock, and the black curve is the fitted conditional mean process by 1-knot NB-based model; Bottom: ACF of the standardized Pearson residuals.}
\label{fig:ericsson_fit}
\end{figure}

Another tool for model checking is through the probability integral transform (PIT). When the underlying distribution is continuous, it is well known that the PIT follows standard uniform distribution. However, if the underlying distribution is discrete, some adjustments are required and the so-called randomized PIT is therefore introduced by perturbing the step function characteristic of the CDF of discrete random variables (see \cite{Brockwell06}). More recently, \cite{Czado} proposed a non-randomized version of PIT as an alternative adjustment. Since it usually gives the same conclusion for model checking, we do not provide the non-randomized version here. For any $t$, the randomized PIT is defined by 
\begin{eqnarray*}
\tilde{u}_t:=F_t(Y_t-1)+\nu_t \bigr[F_t(Y_t)-F_t(Y_t-1)\bigr],
\end{eqnarray*}
where $\{\nu_t\}$ is a sequence of iid uniform $(0,1)$ random variables, $F_t(\cdot)$ is the predictive cumulative distribution. In our situation, $F_t(\cdot)$ is simply the CDF of a Poisson or a negative binomial distribution. If the model is correct, then $\tilde{u}_t$ is an iid sequence of uniform $(0,1)$ random variables. \cite{Jung11} reviewed several ways to depict this and we adopt their method in our study. To test if the PIT follows $(0,1)$ uniform distribution, the histograms of PIT from different models are plotted and a Kolmogorov-Smirnov test is carried out. The results are summarized in Figure \ref{fig:Ericsson_PIT}, and the $p$-values are reported in Table \ref{tab:Ericsson_scores}. It can be seen that both of the two negative binomial-based models pass the PIT test, while none of the Poisson-based models does. This observation could be explained, as mentioned above, by the over-dispersion phenomenon of the data. 
\begin{figure}
\centering
\makebox{\includegraphics[scale=.5]{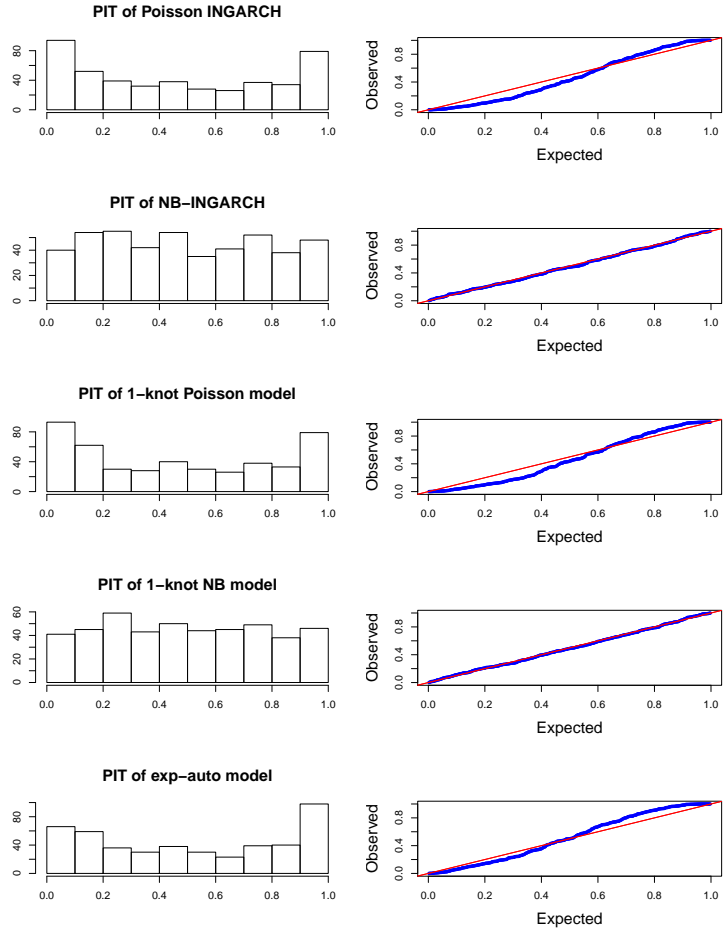}}
\caption{Left: histograms of randomized PIT's for all of the models fitted to the Ericsson stock data; Right: QQ-plots of $\tilde{u}_t$ against standard uniform distribution for the corresponding models, where the straight line is the $45^{\circ}$ line with zero intercept.}
\label{fig:Ericsson_PIT}
\end{figure}

To measure the power of predictions by models, various scoring rules have been proposed in literature, see e.g., \cite{Czado} and \cite{Jung11}. Most of them are computed as the average of quantities related to predictions and take the form $(n-1)^{-1}\sum_{t=2}^n s(F_t(Y_t))$ where $F_t(\cdot)$ is the CDF of the prediction distribution and $s(\cdot)$ denotes some scoring rule. In this paper we calculate three scoring rules: logarithmic score (LS), quadratic score (QS) and ranked probability score (RPS), as a basis for evaluating the relative performance of our fitted models. For definition of these scores, see \cite{Jung11}. Table \ref{tab:Ericsson_scores} summarizes these scores for all of the fitted models. As seen from the table, most of the diagnostic tools favor the one-knot negative binomial model for the Ericsson data.
  \begin{table}
    \caption{\label{tab:Ericsson_scores}Quantitative model checking for Ericsson data}
  \centering
\resizebox{13.5cm}{!}{
\begin{tabular}{| l  c c  c  c  c|}
\hline
Model & log likelihood &$p$-value of PIT & LS & QS & RPS  \\ \hline
Poisson INGARCH & -1433.19 & $<10^{-5}$ & 3.1167 & -0.0576 & 2.6883 \\
NB INGARCH & -1332.02 & 0.7386 & 2.8958 & -0.0671 & 2.6063 \\
1-knot Poisson model &  -1431.21 &  $<10^{-5}$& 3.1123 & -0.0573 & 2.6848 \\ 
2-knot Poisson model  & -1431.08 & $<10^{-5}$ & 3.1121 & -0.0575 & 2.6843 \\
3-knot Poisson model  & -1430.58 & $<10^{-5}$ & 3.1110 & -0.0580 & 2.6779 \\
1-knot NB model  & $\boldsymbol{-1331.34}$ &  0.8494 & $\boldsymbol{2.8942}$ & $\boldsymbol{-0.0671}$ & $\boldsymbol{2.6021}$ \\
Exp-auto model & -1448.69 &  $<10^{-5}$& 3.1504 & $-0.0600$ & 2.6924 \\ \hline
\end{tabular}}
\end{table}
\normalsize

\subsubsection*{2. Return times of extreme events of Goldman Sachs Group (GS) stock}
\medskip

As a second example, we construct a time series based on daily log-returns of Goldman Sachs Group (GS) stock from May 4th, 1999 to March 16th, 2012. We first calculate the hitting times, $\tau_1,\tau_2,\ldots$, for which the log-returns of GS stock falls outside the $0.05$ and $0.95$ quantiles of the data. The discrete time series of interest will be the return (or inter-arrival) times $Y_t=\tau_t-\tau_{t-1}$. If the data are in fact iid, or do not exhibit clustering of large values, then the $Y_t$'s should be independent and geometrically distributed with probability of success $p=0.1$ (\cite{ChangPhD}). Figure \ref{fig:gs_return_times} plots the return times of the stock, and the ACF and histogram of the return times. Note that in order to ameliorate the visual effect of some extremely large observations, the time series is also plotted in the top right panel of Figure \ref{fig:gs_return_times} on a reduced vertical scale, in which it is truncated at 80 and the five observations that are affected are depicted by solid triangles.
\begin{figure}
\centering
\makebox{\includegraphics[scale = .7]{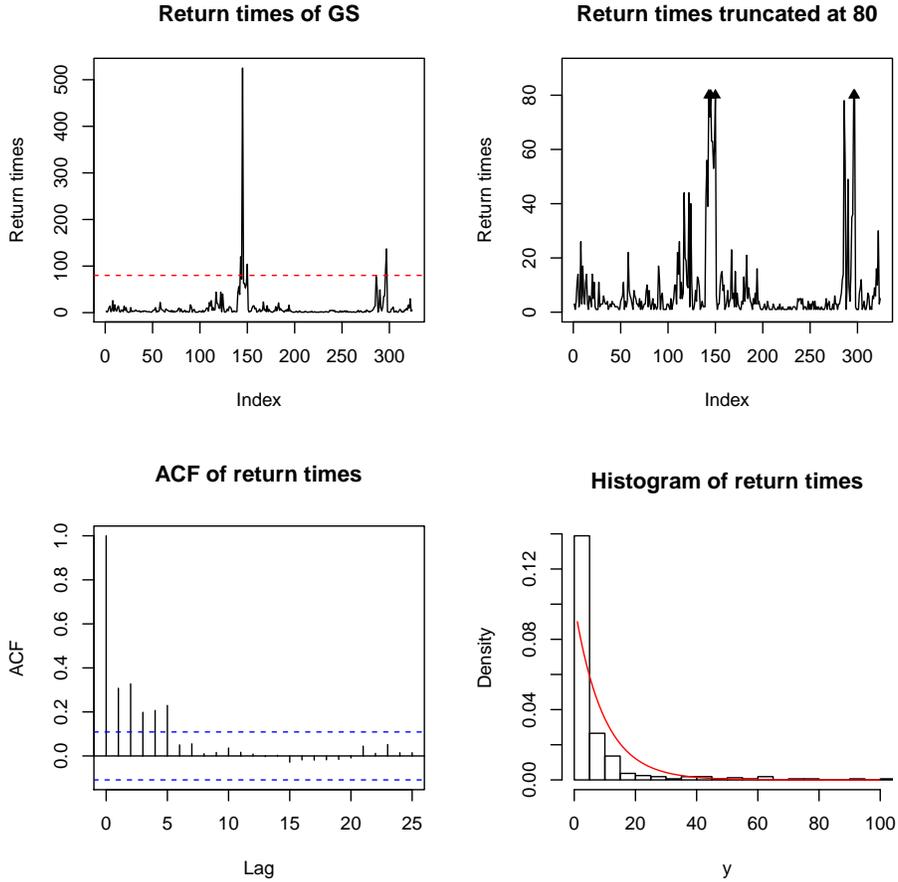}}
\caption{Top left: Return times of GS stock, the dashed horizontal line locates at 80; Top right: Return times truncated at 80 in order to ameliorate the visual effect of the five large observations that are represented by solid triangles; Bottom left: ACF of the return times; Bottom right: Histogram of the return times, where the curve overlaid is the density function of a geometric distribution with $p=0.1$.}
\label{fig:gs_return_times}
\end{figure}

To explore this time series, three models: the geometric INGARCH (negative binomial INGARCH (\ref{eq:nbingarch11}) with $r=1$), and the 1-knot and 2-knot geometric-based models are fitted to the data. The number of knots for the nonlinear dynamic models is chosen by minimizing the AIC, and the locations of knots are estimated by maximizing the likelihood based on a grid search. In addition, the following constraint is imposed: there should be at least 30 observations in each of the regimes segmented by the knots in order to guarantee that there are sufficient observations to obtain quality estimates of the parameters. The sample quantile method for estimating knot locations did not perform as well.

Since it follows from the definition of return times that $Y_t\ge 1$ for any $t$, we use a version of the geometric distribution that counts the total number of trials, instead of only the failures. In particular, the fitted 1-knot geometric-based model is given by $Y_t-1|\mathcal{F}_{t-1}\sim \mbox{Geom}(p_t)$, where
\begin{eqnarray*}
X_t=0.5042+0.4729X_{t-1} + 0.5271(Y_{t-1}-1)-0.0526(Y_{t-1}-5)^+,
\end{eqnarray*}
and the fitted 2-knot geometric-based model is  
\begin{eqnarray*}
X_t=0.5414+0.4531X_{t-1}+0.5469Y_{t-1}-0.2333(Y_{t-1}-9)^++0.2332(Y_{t-1}-18)^+,
\end{eqnarray*}
where $X_t=(1-p_t)/p_t$. Notice that in both models, $\hat{\alpha}+\hat{\beta}$ is very close to unity, i.e., the estimated parameters are close to the boundary of the parameter space. This is similar to the integrated GARCH (IGARCH) model in which $\alpha+\beta=1$. In our application, the mean of the time series of return times is about 10, while the variance is 1101. A simple simulation according to the fitted model yields the mean and median very close to those of the data, but the variance of the simulated data is extraordinarily large, which resembles the feature of the observed data. This is because, although the fitted models are still stationary, the parameters no longer satisfy the conditions specified in Theorem \ref{NonLinearAsympNormal} that ensure a finite variance.

It turns out that the geometric-based models fitted above are capable of capturing the high volatility part of the data. Their standardized Pearson residuals are also calculated and appear to be white. Results of the PIT test are depicted in Figure \ref{fig:exceedance_PIT}, and the prediction scores and the $p$-values of the PIT test are summarized in Table \ref{tab:exceedance_scores}. Two Poisson-based models are also included for comparison, and as expected, they do not perform as well as the geometric-based models.


\begin{figure}
\centering
\makebox{\includegraphics[scale=.5]{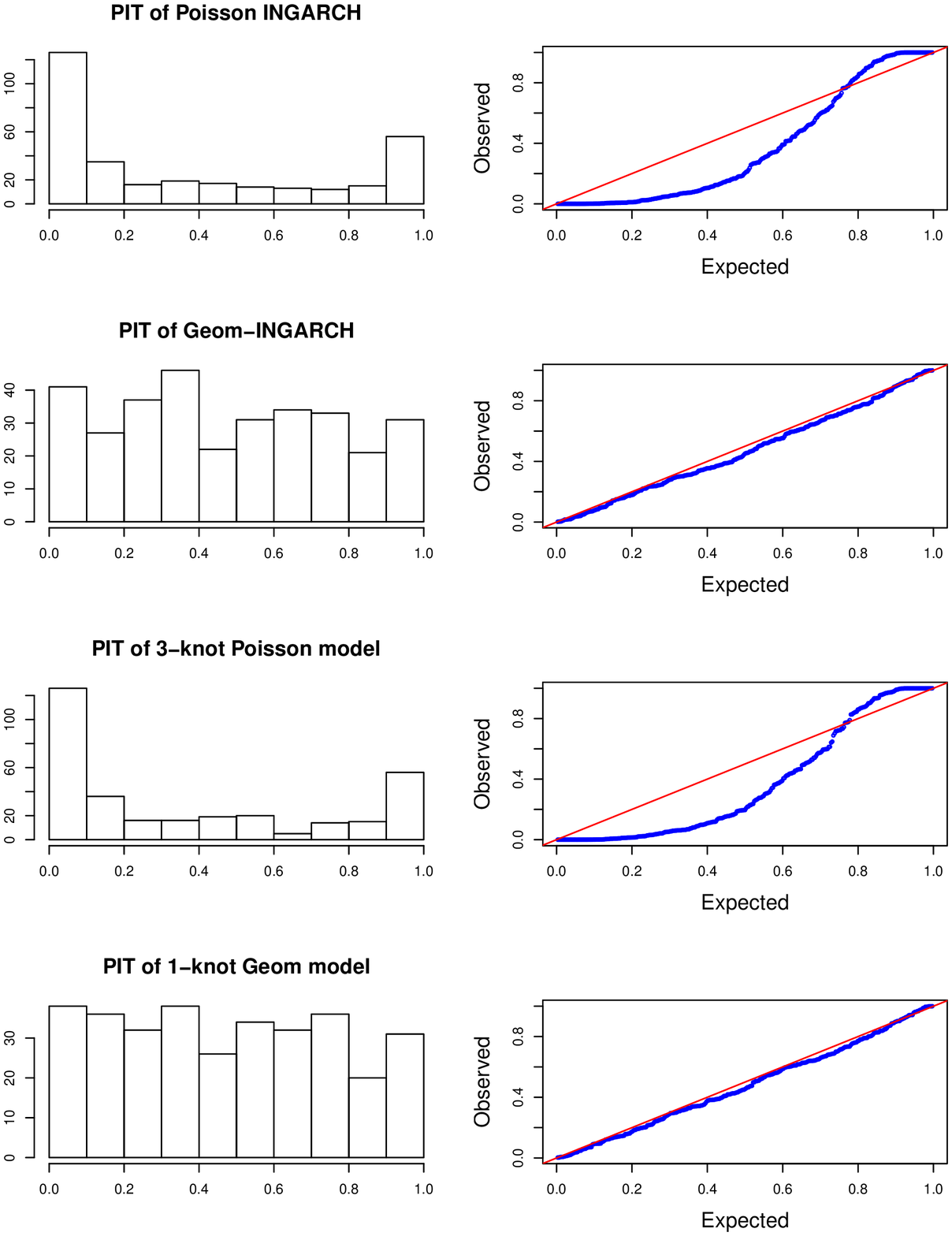}}
\caption{Left: histograms of randomized PIT's for the models fitted to GS return times; Right: QQ-plots of $\tilde{u}_t$ against standard uniform distribution for the corresponding models, where the straight line is the $45^{\circ}$ line with zero intercept.}
\label{fig:exceedance_PIT}
\end{figure}

\begin{table}
    \caption{\label{tab:exceedance_scores}Quantitative model checking for GS return times}
  \centering
\resizebox{13.5cm}{!}{
\begin{tabular}{| l  c c  c  c  c|}
\hline
Model & log likelihood &$p$-value of PIT & LS & QS & RPS  \\ \hline
Poisson INGARCH & -2681.06 & $<10^{-5}$ & 8.2842 & -0.0675 & 4.1373 \\
Geom INGARCH & -857.73 &  0.2581 & 2.6477 & -0.1436 & 3.4100 \\
3-knot Poisson model &  -2670.33 &  $<10^{-5}$ & 8.2510 & -0.0693 & 4.1400 \\ 
1-knot Geom model  & -857.58 & 0.3988 & 2.6472 & $\boldsymbol{-0.1436}$ & 3.4041 \\
2-knot Geom model  & $\boldsymbol{-857.42}$ &  0.2006 & $\boldsymbol{2.6468}$ & -0.1435 & $\boldsymbol{3.3939}$ \\ \hline
\end{tabular}}
\end{table}
\normalsize

\section*{Acknowledgement}
This research is supported in part by NSF grant DMS-1107031.

\section*{Appendix A. Properties of the exponential family}

An important property of the one-parameter exponential family that is heavily used in this paper is the stochastic monotonicity. A random variable $X$ is said to be stochastically smaller than a random variable $Y$ (written as $X\le_{ST}$ Y) if $F(x)\ge G(x)$ for all $x$, where $F(x)$ and $G(x)$ are the cumulative distribution functions of $X$ and $Y$ respectively. We refer readers to \cite{YamingYu} for the related theory.
\begin{prop}
\label{STexponential}
Suppose two random variables $Y'$ and $Y''$ follow distributions belonging to the one-parameter exponential family (\ref{eq:expfamily}) with the same $A, h$ and $\mu$, but with natural parameters $\eta'$ and $\eta''$ respectively. If $\eta'\le \eta''$, then $Y'$ is stochastically smaller than $Y''$.
\end{prop}
\begin{proof}
Denote the probability density functions of $Y'$ and $Y''$ as $p(y|\eta')$ and $p(y|\eta'')$ defined in (\ref{eq:expfamily}), respectively. Then the log ratio of the two densities is
\begin{eqnarray*}
l(y)&=&\log\frac{p(y|\eta')}{p(y|\eta'')}=\log\frac{\exp\{\eta' y-A(\eta')\}h(y)}{\exp\{\eta'' y-A(\eta'')\}h(y)}\\
    &=&y(\eta'-\eta'')+[A(\eta'')-A(\eta')],
\end{eqnarray*}
which is apparently a concave function in $y$. So it follows from Definition 2 in \cite{YamingYu} that $Y'$ is log concave relative to $Y''$, i.e., $Y'\le_{lc} Y''$. Moreover, since $A(\eta)$ is increasing in $\eta$, so $\lim_{y\downarrow 0}l(y)=A(\eta'')-A(\eta')\ge 0$ for continuous $p(y|\eta)$, and $p(0|\eta')/p(0|\eta'')\ge1$ for discrete $p(y|\eta)$. Hence according to Theorem 1 in \cite{YamingYu}, $Y'$ is stochastically smaller than $Y''$, i.e., $Y'\le_{ST} Y''$. 
\end{proof}

Denote $F_x$ as the cumulative distribution function of $p(y|\eta)$ in (\ref{eq:expfamily}) with $x=B(\eta)$, and its inverse $F_x^{-1}(u):=\inf\{t\ge 0: F_x(t)\ge u\}$ for $u\in [0,1]$. The result below provides a useful tool for the coupling technique employed to establish mixing conditions for the observation process.
\begin{prop}
\label{SameTheta}
Suppose that $U$ is a uniform $(0, 1)$ random variable, and define two random variables $Y'$ and $Y''$ as
\begin{eqnarray*}
Y'=F_{x'}^{-1}(U)~~~\mbox{and}~~~ Y''=F_{x''}^{-1}(U),
\end{eqnarray*}
where $x'=B(\eta')$ and $x''=B(\eta'')$. Then $\E |Y'-Y''|=|x'-x''|$.
\end{prop}
\begin{proof}
It follows from the construction of $Y'$ and $Y''$ that they follow the one-parameter exponential family (\ref{eq:expfamily}) with natural parameters $\eta'$ and $\eta''$ respectively, and $\E Y'=x'$, $\E Y''=x''$. If $x'\le x''$, then $Y'$ is stochastically smaller than $Y''$ by virtue of Proposition \ref{STexponential}. It follows that $F_{x'}^{-1}(\theta)\le F_{x''}^{-1}(\theta)$ for $\theta\in (0,1)$, i.e., $Y'\le Y''$. This implies $\E|Y'-Y''|=\E(Y''-Y')=x''-x'$. Similarly if $x'\ge x''$, then $\E|Y'-Y''|=x'-x''$. Hence we have $\E|Y'-Y''|=|x'-x''|$.
\end{proof}

\section*{Appendix B. Proofs}
\subsection*{B.1. Proof of Proposition \ref{modelgmc}}

It suffices to verify the two conditions formulated in \cite{Weibiao04}. For any $y_0$ in the state space $S$, $\E|y_0-f_{u}(y_0)|=\int_0^1 |y_0-g(y_0, F^{-1}_{y_0}(u))|du\le y_0+g(0,0)+a y_0+b\int_0^1 F_{y_0}^{-1}(u)du\le g(0,0)+(1+a+b)y_0<\infty$. Next for a fixed $x_0\in S$, there exists a unique $\eta_0$ such that $x_0=B(\eta_0)$ due to the strict monotonicity of $B(\eta)$. For any $x\ge x_0$, there exists a unique $\eta\ge \eta_0$ such that $x=B(\eta)\ge B(\eta_0)=x_0$. Hence by the contraction condition (\ref{ContractionFunction}), we have
\begin{eqnarray}
\E|X_1(x)-X_1(x_0)|&=&\int_0^1\bigr|g\bigr(x, F_{x}^{-1}(u)\bigr)-g\bigr(x_0, F_{x_0}^{-1}(u)\bigr)\bigr|du \nonumber\\
                          &\le&a|x-x_0|+b\int_0^1\bigr|F_x^{-1}(u)-F_{x_0}^{-1}(u)\bigr|du. \label{eq:gmc2}
\end{eqnarray}
It follows from $x\ge x_0$ and Proposition \ref{STexponential} that for any $u\in (0,1)$, $F_{x_0}^{-1}(u)\le F_x^{-1}(u)$. Therefore
\begin{eqnarray*}
\E|X_1(x)-X_1(x_0)|&\le&a(x-x_0)+b\{\int_0^1 F_x^{-1}(u)du-\int_0^1 F_{x_0}^{-1}(u)du\}\\ 
                              &=&(a+b)(x-x_0).
\end{eqnarray*}
Similarly for $x<x_0$, we have $\E|X_1(x)-X_1(x_0)|\le(a+b)(x_0-x)$. So for any $x\in S$, we have $\E|X_1(x)-X_1(x_0)|\le(a+b)|x-x_0|$. Now suppose  $\E|X_n(x)-X_n(x_0)|\le(a+b)^n|x-x_0|$, then
\small
\begin{eqnarray*}
\E|X_{n+1}(x)-X_{n+1}(x_0)|&=&\E[\E\{|X_{n+1}(X_n(x))-X_{n+1}(X_n(x_0))|\bigr|U_1,\ldots, U_n\}]\\
                              &\le&\E\{(a+b)|X_n(x)-X_n(x_0)|\}\\
                              &\le&(a+b)^{n+1}|x-x_0|.
\end{eqnarray*}
\normalsize
By induction, $\{X_t\}$ is geometric moment contracting and as a result, $\pi$ is its unique stationary distribution. 

To show that $\E_{\pi}X_1<\infty$, notice that by taking conditional expectation on both sides of (\ref{eq:BoundOfG}), we have $\E(X_t|X_{t-1})\le g(0,0)+(a+b)X_{t-1}$. Inductively one can show that for any $t\ge 1$, 
\begin{eqnarray*}
\E(X_t|X_1)\le \frac{1-(a+b)^{t-1}}{1-(a+b)}g(0,0)+(a+b)^{t-1}X_1. 
\end{eqnarray*}
Since for any $x\in S$, $X_t(x)\stackrel{\mathcal{L}}{\longrightarrow}X_1\sim\pi$ as $t\rightarrow \infty$, in particular, $X_t(0)\stackrel{\mathcal{L}}{\longrightarrow}X_1\sim\pi$, so by Theorem 3.4 in \cite{Billingsley99} we have 
\begin{eqnarray*}
\E_{\pi}X_1\le \displaystyle\liminf_{t\rightarrow\infty}\E(X_t|X_1=0)\le \frac{g(0,0)}{1-(a+b)}<\infty.
\end{eqnarray*}
To prove (c), let $\{\xi_t, t\ge 1\}$ be a sequence of independent uniform $(0,1)$ random variables and independent of $\{X_t, t\ge 1\}$, then $Y_t=F_{X_t}^{-1}(\xi_t)$. Since $\{(X_t, \xi_t), t\ge 1\}$ is a stationary sequence if $X_1\sim \pi$, so $\{Y_t, t\ge 1\}$ must also be a stationary process.

\subsection*{B.2. Proof of Proposition \ref{discreteergodicity}}

Define a sequence of functions $\{g_k, k\ge 1\}$ in a way such that $g_1=g$, and for $k\ge 2$, $g_k(x, y_1,\ldots, y_k)=g_{k-1}(g(x, y_k), y_1,\ldots, y_{k-1})$. Then it follows from (\ref{eq:expmodel}) that for all $t\in\mathbb{Z}$, 
\begin{eqnarray*}
X_t=g_k(X_{t-k}, Y_{t-1}, \ldots, Y_{t-k}).
\end{eqnarray*}
By virtue of the contraction condition (\ref{ContractionFunction}), we have $\E\bigr|X_t-g_1(0, Y_{t-1})\bigr|=\E\bigr|g_1(X_{t-1}, Y_{t-1})-g_1(0, Y_{t-1})\bigr|\le a\E X_{t-1}$. By induction, it follows that for any $k\ge 1$,
\begin{eqnarray*}
\E \bigr|X_t-g_k(0, Y_{t-1},\ldots, Y_{t-k})\bigr|\le a^k~\E X_{t-k}.
\end{eqnarray*}
Since $\E_{\pi}X_1<\infty$, it follows that $g_k(0, Y_{t-1},\ldots, Y_{t-k})\stackrel{L^1}{\longrightarrow}X_t$, as $k\rightarrow \infty$. Hence there exists a measurable function $g_{\infty}:\mathbb{N}_0^{\infty}=\{(n_1, n_2, \ldots), n_i\in \mathbb{N}_0\}\longrightarrow [0,\infty)$ such that $X_t=g_{\infty}(Y_{t-1}, Y_{t-2},\ldots)$ almost surely, which proves (a).

To prove (b), denote $\mathcal{F}^{Y}_{k,l}=\sigma\{Y_k, \ldots, Y_l\}$ for $-\infty\le k\le l\le \infty$. Then the coefficients of absolute regularity of the stationary count process $\{Y_t, t\in \mathbb{Z}\}$ are defined as 
\begin{eqnarray*}
\beta(n)=\E\bigr\{\sup_{A\in \mathcal{F}^Y_{n,\infty}}\bigr|P(A|\mathcal{F}^Y_{-\infty, 0})-P(A)\bigr|\bigr\},
\end{eqnarray*}
where $\mathcal{F}^{Y}_{-\infty,0}=\sigma\{X_1, Y_0,Y_{-1}, \ldots\}$ according to $(a)$. Because the distribution of $(Y_n, Y_{n+1}, \ldots)$ given $\sigma\{X_1, Y_0, Y_{-1}, \ldots\}$ is the same as that of $(Y_n, Y_{n+1}, \ldots)$ given $X_1$ for $n\ge 1$, the coefficients of absolute regularity become
\begin{eqnarray}
\beta(n)&=&\E\bigr\{\sup_{A\in \mathcal{F}^{Y}_{n,\infty}}\bigr|P(A|\sigma\{X_1, Y_0,Y_{-1},\ldots\})-P(A)\bigr|\bigr\}\nonumber \\
           &=&\E\bigr\{\sup_{A\in \mathcal{F}^{Y}_{n,\infty}}\bigr|P(A|X_1)-P(A)\bigr|\bigr\}. \label{betacoef}
\end{eqnarray}
Let $\mathcal{B}^{\infty}$ be the $\sigma$-field in $\mathbb{R}^{\infty}$ generated by the cylinder sets, then we can rewrite the coefficients of absolute regularity as 
\begin{eqnarray}
\beta(n)=\E\Bigr\{\sup_{A\in\mathcal{B}^{\infty}}\bigr|P\bigr((Y_n, Y_{n+1},\ldots)\in A|X_1\bigr)-P\bigr((Y_n, Y_{n+1},\ldots)\in A\bigr)\bigr|\Bigr\}.
\label{beta_n}
\end{eqnarray}
We will provide an upper bound for (\ref{beta_n}) by coupling two chains $\{(X_n', Y_n'), n\in \mathbb{Z}\}$ and $\{(X_n'', Y_n''), n\in \mathbb{Z}\}$ defined on a common probability space. Assume that both chains start from the stationary distribution, that is, $X_1'\sim \pi$, $X_1''\sim \pi$ and that $X_1'$ is independent of $X_1''$. Let $\{U_k, k\in \mathbb{Z}\}$ as be an iid sequence of uniform $(0,1)$ random variables, and construct the chains as follows:
\begin{eqnarray*}
&&X_n'=g\bigr(X_{n-1}', F^{-1}_{X_{n-1}'}(U_{n-1})\bigr),~~~ Y_n'=F_{X_n'}^{-1}(U_n),\\
&&X_n''=g\bigr(X_{n-1}'', F^{-1}_{X_{n-1}''}(U_{n-1})\bigr),~~~ Y_n''=F_{X_n''}^{-1}(U_n).
\end{eqnarray*}
Since $X_1'$ and $X_1''$ are independent, so for any $A\in \mathcal{B}^{\infty}$,
\begin{eqnarray*}
P((Y_n'', Y_{n+1}'',\ldots)\in A|X_1')=P((Y_n, Y_{n+1},\ldots)\in A).
\end{eqnarray*}
Hence we have 
\begin{eqnarray}
&&\bigr|P\bigr((Y_n, Y_{n+1},\ldots)\in A|X_1=x\bigr)-P\bigr((Y_n, Y_{n+1},\ldots)\in A\bigr)\bigr| \nonumber \\
&=&\bigr|P\bigr((Y_n', Y_{n+1}',\ldots)\in A|X_1'=x\bigr)-P\bigr((Y_n'', Y_{n+1}'',\ldots)\in A|X_1'=x\bigr)\bigr| \nonumber \\ 
&\le& P\bigr((Y_n',Y_{n+1}',\ldots)\neq (Y_n'', Y_{n+1}'',\ldots)|X_1'=x\bigr). \label{eq:differenceP}
\end{eqnarray}
Therefore the coefficients of absolute regularity are bounded by
\begin{eqnarray}
\beta(n)\le P\bigr((Y_n',Y_{n+1}',\ldots)\neq (Y_n'', Y_{n+1}'',\ldots)\bigr)\le \displaystyle\sum_{k=0}^{\infty}P(Y_{n+k}'\neq Y_{n+k}''). \label{eq:finalbeta}
\end{eqnarray}
Observe that the construction of the two chains agrees with the definition of geometric moment contraction (Definition 1 in \cite{Weibiao04}), so it follows from Proposition \ref{modelgmc} that $\E|X_n'-X_n''|\le (a+b)^n$ for all $n$. Then
\begin{eqnarray*}
P(Y_n'\neq Y_n'')&=&\E\{P(Y_n'\neq Y_n''|X_n,X_n'')\}= \E\{P(|Y_n'-Y_n''|\ge 1|X_n, X_n'')\} \\
                       &\le&\E\{E|Y_n'-Y_n''|\bigr|X_n',X_n'')\}=\E|X_n'-X_n''|\le (a+b)^n.
\end{eqnarray*}
Hence according to (\ref{eq:finalbeta}), the coefficients of absolute regularity satisfy $\beta(n)\le \sum_{k=0}^{\infty}(a+b)^{n+k}=(a+b)^n/(1-(a+b))$. Recall the well-known fact that $\beta$-mixing implies strong mixing (e.g., \cite{Doukhan94}), so $\{Y_t, t\ge 1\}$ is stationary and strongly mixing at geometric rate, in fact, it is ergodic. In particular, $\{Y_t, t\ge 1\}$ is an ergodic stationary process. It follows from $X_t=g_{\infty}(Y_{t-1}, Y_{t-2},\ldots)$ that $\{X_t, t\ge 1\}$ is also ergodic.

\subsection*{B.3. Proof of Proposition \ref{ContinuousErgodicity}}
The proof utilizes the classic Markov chain theory, see for example \cite{MeynTweedie}. (a) follows from the same argument as in the proof of Proposition \ref{discreteergodicity}. As for (b), for any fixed $\epsilon>0$, define $\phi$ as Lebesgue measure on $[x^{\ast}, \infty)$, where $x^{\ast}=(g(0,0)+b\epsilon)/(1-a)$, and let $A$ be a set with $\phi(A)>0$. To prove the $\phi-$irreducible, we need to show that for any $x_1\in S$, there exists $n\ge 1$, such that $P^n(x_1, A)>0$. If $x_1<x^{\ast}$, then $g(x_1, \epsilon)<g(0,0)+ax_1+b\epsilon\le x^{\ast}$, which implies that $\phi\bigr(A\cap[g(x_1, \epsilon),\infty)\bigr)>0$. Because of the assumptions on the function $g$, and the fact that the distribution of $Y_1$ given $X_1=x_1$ has positive probability everywhere, so $P(x_1, A)>0$. On the other hand, if $x_1\ge x^{\ast}$, it is easy to see that $g(x_1, \epsilon/2)\le g(x_1, \epsilon)\le x_1$. If $g(x_1, \epsilon/2)<x^{\ast}$, then by the same argument above, we have $P(x_1, A)>0$. However, if $g(x_1, \epsilon/2)\ge x^{\ast}$, then $ag(x_1,\epsilon/2)+b\epsilon\le g(x_1,\epsilon/2)-g(0,0)\le ax_1+b\epsilon/2$. Hence we have $x^{\ast}\le g(x_1, \epsilon/2)\le x_1-(b\epsilon)/(2a)$. By induction, there exists $n\ge 1$ such that $g(x_n, \epsilon/2)\le x_1-n(b\epsilon)/(2a)<x^{\ast}$, where $x_t=g(x_{t-1},\epsilon/2)$ for $t=1,\ldots,n$. Since $\epsilon>0$, and the function $g$ is increasing in both coordinates, so $P^{n+1}(x_1, A)>0$. Hence $\{X_t, t\ge 1\}$ is $\phi-$irreducible.
 
We now show that $\{X_t, t\ge 1\}$ is aperiodic, i.e., a $\phi-$irreducible Markov chain is said to be aperiodic if there exists a small set $A$ with $\phi(A)>0$ such that for any $x\in A$, $P(x, A)>0$ and $P^2(x, A)>0$. Note that in the setting of the proposition, any compact set is a small set. So we take $A=[x^{\ast}, K]$ for some positive $K$ large enough. For any $x_1\in A$, from the proof of $\phi-$irreducibility, it is easy to see that $P(x_1, A)>0$. Similarly we have $P^2(x, A)=P(X_2\in A|X_0=x)\ge P(X_2\in A|X_1\in A)P(X_1\in A|X_0=x)>0$.
 
To check the drift condition, let $V(x)=1+x$. There exists $\delta>0$, such that $a+b<1-\delta$. For $x\ge (g(0,0)+\delta)/(1-a-b-\delta)$, we have 
\begin{eqnarray*}
\E \{V(X_1)|X_0=x\}&=&\E (1+X_1|X_0=x)=1+\E\{g(x, Y_0)|X_0=x\}\\
                            &\le& 1+g(0,0)+(a+b)x\le(1-\delta)(1+x)=(1-\delta)V(x).
\end{eqnarray*}
Hence the drift condition holds by taking the small set $A=[x_0^{\ast}, \{g(0,0)+\delta\}/(1-a-b-\delta)]$, which establishes the geometric ergodicity of $\{X_t\}$. It is well known that a geometrically ergodic Markov chain starting from its stationary distribution is strongly mixing with geometrically decaying rate, hence is an ergodic stationary time series (e.g., \cite{MeynTweedie}). Denote $\{\xi_t, t\ge 1\}$ as a sequence of iid uniform $(0,1)$ random variables, then it follows from $Y_t=F_{X_t}^{-1}(\xi_t)$ that $\{Y_t, t\ge 1\}$ is stationary and ergodic.

\subsection*{B.4. Proof of Theorem \ref{Consistency}}
We first show the identifiability and then establish the consistency result using Lemma \ref{WaldConsistency}. Throughout the proof, we assume that the process $\{(Y_t, X_t), t\in \mathbb{Z}\}$ is in its stationary regime. Note that by assumption (A1), $X_t(\theta)\ge x_{\theta}^{\ast}\in\mathcal{R}(B)$, which implies $\eta_t(\theta)\ge B^{-1}(x_{\theta}^{\ast})$. So it follows from assumptions (A2) and (A4) that for any $\theta\in \Theta$,
\begin{eqnarray*}
\E l_t(\theta)&=&\E\bigr\{Y_t B^{-1}(X_t(\theta))-A\bigr(B^{-1}(X_t(\theta))\bigr)\bigr\}\\
                  &\le&\E\bigr\{Y_t\displaystyle\sup_{\theta\in \Theta}B^{-1}(X_t(\theta))\bigr\}-A((B^{-1}(x^{\ast}_{\theta}))<\infty.
\end{eqnarray*}
This implies $\E l_t^{+}(\theta)<\infty$. Denote $M_n(\theta)=\sum_{t=1}^n l_t(\theta)/n$, then $M_n(\theta)\stackrel{a.s.}{\longrightarrow}M(\theta)=\E\bigr\{Y_1\eta_1(\theta)-A(\eta_1(\theta))\bigr\}$ according to the extended mean ergodic theorem (see \cite{billingsley95} pp. 284 and 495). In order to prove the identifiability, we need to show that $\theta_0$ is the unique maximizer of $M(\theta)$, that is, for any $\theta\in \Theta\setminus\{\theta_0\}$, $M(\theta)-M(\theta_0)<0$. First it follows from assumption (A5) that for any $\theta\neq \theta_0$ and all $t$, $P_{\theta_0}(G_t(\theta,\theta_0))>0$, where $G_t(\theta,\theta_0)=\{X_t(\theta)\neq X_t(\theta_0)\}$. Let $G=G_t(\theta,\theta_0)$, then we have 
\begin{eqnarray*}
M(\theta)-M(\theta_0)&=&\E\bigr[Y_t\bigr\{B^{-1}(X_t(\theta))-B^{-1}\bigr(X_t(\theta_0)\bigr)\bigr\}\\
                                   &&-\bigr\{A(B^{-1}(X_t(\theta)))-A(B^{-1}(X_t(\theta_0)))\bigr\}\bigr]\\
                                   &=&\E\bigr[X_t(\theta_0)\bigr\{B^{-1}(X_t(\theta))-B^{-1}\bigr(X_t(\theta_0)\bigr)\bigr\}\\
                                   &&-\bigr\{A(B^{-1}(X_t(\theta)))-A(B^{-1}(X_t(\theta_0)))\bigr\}\bigr]\\
                                  &=&\int_GX_t(\theta_0)\bigr\{B^{-1}(X_t(\theta))-B^{-1}\bigr(X_t(\theta_0)\bigr)\bigr\}\\
                                   &&-\bigr\{A(B^{-1}(X_t(\theta)))-A(B^{-1}(X_t(\theta_0)))\bigr\}dP_{\theta_0}.            
\end{eqnarray*}
On the set $G$, there exists $c\in \mathbb{R}$ between $B^{-1}\bigr(X_t(\theta)\bigr)$ and $B^{-1}\bigr(X_t(\theta_0)\bigr)$ such that $A(B^{-1}(X_t(\theta)))-A(B^{-1}(X_t(\theta_0)))=B(c)\{B^{-1}(X_t(\theta))-B^{-1}(X_t(\theta_0))\}$ by the mean value theorem. It follows from $A''(\eta)>0$ that $A(\eta)$ is strictly convex and $c$ must be strictly between $B^{-1}(X_t(\theta))$ and $B^{-1}(X_t(\theta_0))$. So there exists $\xi\in\mathbb{R}$ lying strictly between $X_t(\theta)$ and $X_t(\theta_0)$ such that $\xi=B(c)$. Therefore
\begin{eqnarray*}
M(\theta)-M(\theta_0)=\int_G (X_t(\theta_0)-\xi)\{B^{-1}(X_t(\theta))-B^{-1}(X_t(\theta_0))\}dP_{\theta_0}.
\end{eqnarray*}
Since $B(\eta)$ is strictly increasing, so $(X_t(\theta_0)-\xi)\{B^{-1}(X_t(\theta))-B^{-1}(X_t(\theta_0))\}<0$ in either of the two cases: $X_t(\theta)<X_t(\theta_0)$ and $X_t(\theta)>X_t(\theta_0)$. Hence $M(\theta)-M(\theta_0)<0$, for any $\theta\neq \theta_0$, which establishes the identifiability.
To show the consistency, first note that by assumption (A4), we have
\begin{eqnarray*}
\E\displaystyle\sup_{\theta\in \Theta}l_t(\theta)&=&\E\{Y_t \sup_{\theta\in \Theta}B^{-1}(X_t(\theta))-\inf_{\theta\in\Theta}A(B^{-1}(X_t(\theta)))\}\\
                                                          &\le&\E\{Y_t\displaystyle\sup_{\theta\in \Theta}B^{-1}(X_t(\theta))\}-A(B^{-1}(x^{\ast}))<\infty.
\end{eqnarray*}
The function $f_{\theta}$ in Lemma \ref{WaldConsistency} can be defined as 
\begin{eqnarray*}
f_{\theta}(\yy)=y_1B^{-1}(g_{\infty}^{\theta}(y_0,y_{-1},\ldots))-A(B^{-1}(g_{\infty}^{\theta}(y_0, y_{-1},\ldots))),
\end{eqnarray*}
where $\yy=(y_1, y_0, y_{-1},\ldots)$. Hence it follows from assumption (A2) and Lemma \ref{WaldConsistency} that $M(\theta)$ is upper-semicontinuous and for any compact subset $K\subset \Theta$, $\limsup_{n\rightarrow\infty}\sup_{\theta\in K}M_n(\theta)\le \sup_{\theta\in K}M(\theta)$. Take $\mathcal{U}_0$ as a local base of $\theta_0$ and let $U\in\mathcal{U}_0$ be a neighborhood of $\theta_0$, then Lemma \ref{WaldConsistency} can be applied to $\Theta\setminus U$. Because u.s.c function attains its maximum on compact sets and $M(\theta)<M(\theta_0)$ for any $\theta\neq \theta_0$, we have
\begin{eqnarray}
\displaystyle\limsup_{n\rightarrow\infty}\sup_{\theta\in \Theta\setminus U}M_n(\theta)\le \sup_{\theta\in\Theta\setminus U}M(\theta)<M(\theta_0),~~~P_{\theta_0}\mbox{-a.s.} \label{usc1}
\end{eqnarray}
Notice that for any $\tilde{\theta}\notin U$, $M_n(\tilde{\theta})\le \sup_{\theta\in\Theta\setminus U}M_n(\theta)$. Let $\omega\in\Omega$ such that (\ref{usc1}) holds and $M(\theta_0)=\lim_{n\rightarrow\infty}M_n(\theta_0)$. For such $\omega$, suppose $\hat{\theta}_n\notin U$ infinitely often, say, along a sequence denoted by $\widetilde{\mathbb{N}}$, then
\begin{eqnarray}
\displaystyle\liminf_{n\rightarrow\infty} M_n(\hat{\theta}_n)&\le& \liminf_{n\rightarrow\infty, n\in\widetilde{\mathbb{N}}}M_n(\hat{\theta}_n)\le \limsup_{n\rightarrow\infty, n\in\widetilde{\mathbb{N}}}M_n(\hat{\theta}_n) \nonumber\\
                                                                               &\le& \limsup_{n\rightarrow\infty, n\in\widetilde{\mathbb{N}}}\sup_{\theta\notin U}M_n(\theta)\le \limsup_{n\rightarrow\infty}\sup_{\theta\notin U}M_n(\theta). \label{usc2}
\end{eqnarray}
However, according to (\ref{usc1}), we have
\begin{eqnarray*}
\displaystyle\limsup_{n\rightarrow\infty}\sup_{\theta\in\Theta\setminus U}M_n(\theta)\le \sup_{\theta\in\Theta\setminus U}M(\theta)<M(\theta_0)=\lim_{n\rightarrow\infty} M_n(\theta_0)\le \liminf_{n\rightarrow\infty}M_n(\hat{\theta}_n),
\end{eqnarray*}
which contradicts (\ref{usc2}). Hence there exists a null-set $N_U$ such that for all $\omega\notin N_U$, $\hat{\theta}_n\in U$ for all $n$ large enough. It follows by taking any set $U\in \mathcal{U}_0$ that $\hat{\theta}_n$ converges to $\theta_0$ almost surely.

\subsection*{B.5. Proof of Theorem \ref{AsympNormal}}
We define a linearized form of $\eta_t(\theta)$ as $\eta_t^\dagger(\theta):=\eta_t(\theta_0)+(\theta-\theta_0)^T\dot{\eta}_t$, and the corresponding linearized log-likelihood function of $l(\theta)$ as
\begin{eqnarray*}
l^{\dagger}(\theta):=\displaystyle\sum_{t=1}^n \eta_t^{\dagger}(\theta)Y_t-\sum_{t=1}^n A(\eta_t^{\dagger}(\theta)).
\end{eqnarray*}
Let $u=\sqrt{n}(\theta-\theta_0)$, then define
\small
\begin{eqnarray}
R_n^{\dagger}(u)&=& l^{\dagger}(\theta_0)-l^{\dagger}(\theta_0+u n^{-1/2}) \nonumber \\
                         &=&\displaystyle\sum_{t=1}^n Y_t \eta_t-\sum_{t=1}^n A(\eta_t)-\sum_{t=1}^n (\eta_t+u^T n^{-1/2}\dot{\eta_t})Y_t+\sum_{t=1}^n A(\eta_t+u^Tn^{-1/2}\dot\eta_t)  \nonumber \\
                           &=&-u^T n^{-1/2} \sum_{t=1}^n Y_t \dot{\eta_t}+\sum_{t=1}^n \{A(\eta_t+u^Tn^{-1/2}\dot{\eta_t})-A(\eta_t)\} \nonumber \\
                           &=&-u^T n^{-1/2} \sum_{t=1}^n \{Y_t-B(\eta_t)\}\dot{\eta_t} \nonumber \\
&&+\sum_{t=1}^n \{A(\eta_t+u^Tn^{-1/2}\dot{\eta_t})-A(\eta_t)-u^Tn^{-1/2}B(\eta_t)\dot{\eta_t}\}. \label{eq:LinearR}
\end{eqnarray}
\normalsize
Let $s_t=n^{-1/2}\{Y_t-B(\eta_t)\}\dot{\eta_t}$, then $\E(s_t|\mathcal{F}_{t-1})=n^{-1/2}\E[\{Y_t-B(\eta_t)\}\dot{\eta_t}|\mathcal{F}_{t-1}]=0$, so $\{s_t,t \ge 1\}$ is a martingale difference sequence. Note that 
\begin{eqnarray*}
\displaystyle\sum_{t=1}^n \E(s_ts_t^T|\mathcal{F}_{t-1})&=&\frac{1}{n}\sum_{t=1}^n \E[\{Y_t-B(\eta_t)\}^2 \dot{\eta_t}\dot{\eta_t}^T|\mathcal{F}_{t-1}]\\
&=&\frac{1}{n}\sum_{t=1}^n B'(\eta_t)\dot{\eta_t}\dot{\eta_t}^T,
\end{eqnarray*}
which converges almost surely to $\Omega$ by the mean ergodic theorem and assumption (A7). Moreover, for any $\epsilon>0$, 
\begin{eqnarray*}
&&\sumn \E\{s_t s_t^T \mathbf{1}_{[|s_t|\ge \epsilon]}|\mathcal{F}_{t-1}\}\\
                     &=&1/n\sumn \dot{\eta_t} \dot{\eta_t}^T \E[\{Y_t-B(\eta_t)\}^2\mathbf{1}_{[|\{Y_t-B(\eta_t)\}\dot{\eta_t}|\ge \epsilon\sqrt{n}]}|\mathcal{F}_{t-1}]\\
                     &\le&1/n\sumn \dot{\eta_t} \dot{\eta_t}^T \E[\{Y_t-B(\eta_t)\}^2\mathbf{1}_{[|\{Y_t-B(\eta_t)\}\dot{\eta_t}|\ge M]}|\mathcal{F}_{t-1}]\\
                     &\longrightarrow& \E[\{Y_1-B(\eta_1)\}^2\dot{\eta_1}\dot{\eta_1}^T\mathbf{1}_{[|\{Y_t-B(\eta_t)\}\dot{\eta_t}|\ge M]}]~~\mbox{as}~n\rightarrow \infty\\
                     &\longrightarrow&0~~ \mbox{as}~M\rightarrow 0.
\end{eqnarray*}
\normalsize
Then it follows from the central limit theorem for martingale difference sequences that
\begin{eqnarray*}
\displaystyle\sum_{t=1}^n s_t\stackrel{\mathcal{L}}{\longrightarrow} V\sim N(0, \Omega),~~~\mbox{as}~~~n\rightarrow\infty,
\end{eqnarray*} 
where $\Omega$ is evaluated at $\theta_0$. The other term in (\ref{eq:LinearR}) by Taylor expansion is 
\begin{eqnarray*}
\frac{1}{2n}\sumn u^T \{B'(\eta_t)\dot{\eta_t}\dot{\eta_t}^T\} u+\mathcal{O}_p(n^{-3/2}\sumn B''(\eta_t)(u^T\dot{\eta_t})^3),
\end{eqnarray*}
which is of the order of $u^T\Omega u/2+o_P(1)$. Hence $R_n^{\dagger}(u)\indist -u^T V+\frac{1}{2}u^T \Omega u$, where $V\sim N(0, \Omega)$. It then follows that $\argmin_u R_n^{\dagger}(u) \indist \argmin_u\{-u^T V+\frac{1}{2}u^T \Omega u\}=\Omega^{-1}V\sim N(0, \Omega^{-1})$.

For the rest of the proof, we show that the difference between $R_n(u):=l(\theta_0)-l(\theta_0+un^{-1/2})$ and $R_n^{\dagger}(u)$ is negligible as $n$ grows large. By writing $\theta=\theta_0+un^{-1/2}$, the difference becomes
\begin{eqnarray}
R_n^{\dagger}(u)-R_n(u)&=& \sumn \{Y_t-B(\eta_t)\}\{\eta_t(\theta)-\eta_t-u^Tn^{-1/2}\dot{\eta_t}\} \nonumber \\
                                       &&-\sumn [A(\eta_t(\theta))-A(\eta_t+u^Tn^{-1/2}\dot{\eta_t}) \nonumber\\
                                       &&-B(\eta_t)\{\eta_t(\theta)-\eta_t-u^Tn^{-1/2}\dot{\eta_t}\}].  \label{eq:DifferenceR}
\end{eqnarray}
By Taylor expansion, the first term in (\ref{eq:DifferenceR}) is $1/(2n)\sum_{t=1}^n \{Y_t-B(\eta_t)\}u^T \ddot{\eta_t}(\theta_t^{\ast})u=1/(2n)u^T\newline[\sum_{t=1}^n \{Y_t-B(\eta_t)\}\ddot{\eta_t}+\sum_{t=1}^n \{Y_t-B(\eta_t)\}\{\ddot{\eta_t}(\theta_t^{\ast})-\ddot{\eta_t}\}]u$, where $\theta_t^{\ast}$ lies between $\theta$ and $\theta_0$, and $\ddot{\eta_t}=\partial^2\eta_t/\partial \theta\partial\theta^T$. Since 
\begin{eqnarray*}
\frac{1}{n}\sum_{t=1}^n \{Y_t-B(\eta_t)\}\ddot{\eta_t} &\asconv& \E[\{Y_t-B(\eta_t)\}\ddot{\eta_t}]\\
&=&\E[\ddot{\eta_t}\E\{Y_t-B(\eta_t)|\mathcal{F}_{t-1}\}]=0,
\end{eqnarray*} 
and $1/n\sum_{t=1}^n \{Y_t-B(\eta_t)\}\{\ddot{\eta_t}(\theta_t^{\ast})-\ddot{\eta_t}\} \asconv 0$ under the smoothness assumption, so the first term in (\ref{eq:DifferenceR}) converges to 0 uniformly on $[-K, K]$ for any $K>0$. We now apply Taylor expansion to each component in the second term of (\ref{eq:DifferenceR}),
\begin{eqnarray*}
&&A(\eta_t(\theta))=A(\eta_t)+u^Tn^{-1/2}B(\eta_t)\dot{\eta}_t\\
&&~~~~~~~~~~~~~~~+\frac{1}{2n}u^T\{B(\eta_t(\theta_1^{\ast}))\ddot{\eta}_t(\theta_1^{\ast})+B'(\theta_1^{\ast})\dot{\eta}_t(\theta_1^{\ast})\dot{\eta}_t(\theta_1^{\ast})^T\}u,\\
&&A(\eta_t+u^Tn^{-1/2}\dot{\eta_t})=A(\eta_t)+B(\eta_t)u^Tn^{-1/2}\dot{\eta_t}+\frac{1}{2n}u^TB'(c)\dot{\eta_t}\dot{\eta_t}^Tu,\\
&&\eta_t(\theta)=\eta_t(\theta_0+un^{-1/2})=\eta_t+\dot{\eta_t}u^Tn^{-1/2}+\frac{1}{2n}u^T\ddot{\eta_t}(\theta_2^{\ast})u,
\end{eqnarray*}
where $0\le c\le u^Tn^{-1/2}\dot{\eta_t}$, $\theta_1^{\ast}$ and $\theta_2^{\ast}$ both lie between $\theta_0$ and $\theta$. Therefore the second term in (\ref{eq:DifferenceR}) becomes
\small
\begin{eqnarray*}
&&\sumn [A(\eta_t(\theta))-A(\eta_t+u^Tn^{-1/2}\dot{\eta_t})-B(\eta_t)\{\eta_t(\theta)-\eta_t-u^Tn^{-1/2}\dot{\eta_t}\}]\\
&=&\sumn [A(\eta_t)+u^Tn^{-1/2}B(\eta_t)\dot{\eta_t}+\frac{1}{2n}u^T\{B(\eta_t(\theta_1^{\ast}))\ddot{\eta}_t(\theta_1^{\ast})+B'(\theta_1^{\ast})\dot{\eta}_t(\theta_1^{\ast})\dot{\eta}_t(\theta_1^{\ast})^T\}u\\
&&-A(\eta_t)-B(\eta_t)u^Tn^{-1/2}\dot{\eta_t}-\frac{1}{2n}u^TB'(c)\dot{\eta_t}\dot{\eta_t}^Tu-B(\eta_t)\frac{1}{2n}u^T\ddot{\eta_t}(\theta_2^{\ast})u]\\
&=&\frac{1}{2n}u^T\sumn[\{B(\eta_t(\theta_1^{\ast}))\ddot{\eta_t}(\theta_1^{\ast})-B(\eta_t)\ddot{\eta_t}(\theta_2^{\ast})\}+\{B'(\theta_1^{\ast})\dot{\eta_t}(\theta_1^{\ast})\dot{\eta_t}(\theta_1^{\ast})^T\\
&&-B'(c)\dot{\eta_t}\dot{\eta_t}^T\}]u,
\end{eqnarray*}
\normalsize
which converges to 0 on a compact set of $u$ under smoothness assumptions. So (\ref{eq:DifferenceR}) converges to 0 as $n\rightarrow\infty$, which implies that $\argmin_u R_n(u)$ and $\argmin_u R_n^{\dagger}(u)$ have the same asymptotic distribution, i.e.,
\begin{eqnarray*}
\displaystyle\argmin_u R_n(u)\stackrel{\mathcal{L}}{\longrightarrow}\Omega^{-1}V\sim N(0, \Omega^{-1}).
\end{eqnarray*}
Note that $\argmin_u R_n(u)=\argmax_u~l(\theta_0+un^{-1/2})=\sqrt{n}(\hat{\theta}_n-\theta_0)$, where $\hat{\theta}_n$ is the conditional maximum likelihood estimator. Hence
\begin{eqnarray*}
\sqrt{n}(\hat{\theta}_n-\theta_0) \indist N(0, \Omega^{-1}),~~~\mbox{as}~~n\rightarrow\infty.
\end{eqnarray*}

\subsection*{B.6. Proof of Theorem \ref{LinearAsymp}}
According to Theorems \ref{Consistency} and \ref{AsympNormal}, it is sufficient to establish the identifiability of the model, that is, we need to verify assumption (A5). Suppose for some $t\in\mathbb{Z}$, $X_t(\theta)=X_t(\theta_0)$, $P_{\theta_0}$-a.s, then $\delta+\alpha X_{t-1}(\theta)+\beta Y_{t-1}=\delta_0+\alpha_0 X_{t-1}(\theta_0)+\beta_0 Y_{t-1}$. It follows from (\ref{eq:InfinitePastRep}) that
\small
\begin{eqnarray*}
(\beta-\beta_0)Y_{t-1}=\delta_0-\delta+\alpha_0\bigr(\frac{\delta_0}{1-\alpha_0}+\beta_0\displaystyle\sum_{k=0}^{\infty}\alpha_0^k Y_{t-k-2}\bigr)-\alpha\bigr(\frac{\delta}{1-\alpha}+\beta\displaystyle\sum_{k=0}^{\infty}\alpha^k Y_{t-k-2}\bigr).
\end{eqnarray*}
\normalsize
If $\beta\neq \beta_0$, then $Y_{t-1}\in \mbox{span}\{Y_{t-2}, Y_{t-3},\ldots\}$ which contradicts the fact that $\var(Y_{t-1}|\mathcal{F}_{t-2})>0$. So $\beta$ must be the same as $\beta_0$. Similarly one can show that $\alpha=\alpha_0$ and $\delta=\delta_0$, which implies $\theta=\theta_0$. Hence the model is identifiable.

\subsection*{B.7. Proof of Remark \ref{L2Remark}}
The most difficult case is the derivative with respect to $\theta_2=\alpha$ and we only give its proof, since the arguments for $\delta$ and $\beta$ are similar. First note that 
\begin{eqnarray*}
\E\{B'(\eta_1(\theta_0))\bigr(\frac{\partial \eta_1(\theta_0)}{\partial \alpha}\bigr)^2\}=\E\{\frac{1}{B'(\eta_1)}\bigr(\frac{\partial B(\eta_1)}{\partial \alpha}\bigr)^2\}\le\frac{1}{\underline{c}}\E\{\frac{\partial B(\eta_1)}{\partial\alpha}\}^2,
\end{eqnarray*}
where $\partial B(\eta_1)/\partial \alpha=\delta/(1-\alpha)^2+\beta\sum_{k=1}^{\infty}k \alpha^{k-1}Y_{-k}$. Then on account of stationarity, one can show that
\begin{eqnarray*}
\E\bigr(\sum_{k=1}^{\infty}k \alpha^{k-1}Y_{-k}\bigr)^2&\le& \{\gamma_Y(0)+\frac{2\gamma_Y(1)}{1-\alpha(\alpha+\beta)}\}\displaystyle\sum_{k=1}^{\infty}k^2 \alpha^{2k-2}\\
&&+\frac{2\alpha\gamma_Y(1)}{1-\alpha^2(\alpha+\beta)^2}\sum_{k=1}^{\infty}k \alpha^{2k-2}+\mu^2\bigr(\sum_{k=1}^{\infty}k \alpha^{k-1}\bigr)^2<\infty,
\end{eqnarray*}
where $\mu=\E Y_t<\infty$. Hence $\E[B'(\eta_1(\theta_0))\{\partial \eta_1(\theta_0)/\partial \alpha\}^2]<\infty$ if $\gamma_Y(0)<\infty$.

\subsection*{B.8. Proof of Proposition \ref{poissonpq}}
The proof considers two separate cases: $q=1$ and $q>1$, since they require different methods to construct the state space. 
\begin{enumerate}
\item $q=1$: without loss of generality we consider $p=2$. Denote $\XX_t=(\lambda_t,\lambda_{t+1})$, then $\XX_t$ is a Markov chain. Note that $\lambda_t\ge \lambda^{\ast}=\delta/(1-\alpha_1-\alpha_2)$. $\XX_t$ can be constructed by iteratively imposing the random function $f_u$, $u\in (0,1)$,
\begin{eqnarray*}
f_u: [\lambda^{\ast},\infty)\times[\lambda^{\ast},\infty) &\longrightarrow& [\lambda^{\ast},\infty)\times[\lambda^{\ast},\infty)\\
                                                             \xx=(\lambda_1,\lambda_2) &\longmapsto& (\lambda_2,\delta+\alpha_1\lambda_2+\alpha_2\lambda_1+\beta F_{\lambda_2}^{-1}(u)).
\end{eqnarray*}
For any $\xx=(x_1,x_2), \yy=(y_1,y_2)$ in the state space $S=[\lambda^{\ast},\infty)\times[\lambda^{\ast},\infty)$, define metric $\rho$ as $\rho(\xx, \yy)=w_1|x_1-y_1|+w_2|x_2-y_2|$, where $w_i >0, i=1,2$ and $w_1, w_2$ are to be decided. Let $\xx_1=(\lambda_1^0,\lambda_2^0):=(\lambda^{\ast},\lambda^{\ast})$, then for any $\xx=(\lambda_1,\lambda_2)$ we have
\begin{eqnarray}
\E\rho(\XX_1(\xx),\XX_1(\xx_1))&=&\int_0^1 \rho(f_u(\xx),f_u(\xx_1))du \nonumber\\
                                     &=&a_2w_2|\lambda_1-\lambda_1^0|+\{w_1+w_2(a_1+b)\}|\lambda_2-\lambda_2^0|, \nonumber
\end{eqnarray}
where the last equation holds because $\lambda_t\ge \lambda^{\ast}$. Therefore it is sufficient to find an $r\in (0,1)$ and strictly positive $(w_1, w_2)$ such that
\begin{eqnarray*}
\E\rho(\XX_1(\xx),\XX_1(\xx_1))\le r\rho(\xx,\xx_1)=r\{w_1|\lambda_1-\lambda_1^0|+w_2|\lambda_2-\lambda_2^0|\}.
\end{eqnarray*}
This can be obtained if the equation $r^2-(a_1+b)r-a_2=0$ yields a root $r_{+}=\frac{a_1+b+\sqrt{(a_1+b)^2+4a_2}}{2}<1$. It can be shown that under $\alpha_1+\alpha_2+\beta<1$ the root $r_{+}\in (0, 1)$. Note that the choice of $(w_1, w_2)$ is not unique.\\
\item $q>1$: without loss of generality we consider the INGARCH(2,2) model. Define a Markov chain $\XX_t=(Y_t,\lambda_t,\lambda_{t+1})$, then the chain can be obtained by defining the iterated random functions $f_u: \mathbb{Z}_0\times[\lambda^{\ast},\infty)\times[\lambda^{\ast},\infty)\rightarrow \mathbb{Z}_0\times[\lambda^{\ast},\infty)\times[\lambda^{\ast},\infty)$ as $f(\xx)=f(n,\lambda_1,\lambda_2)=(F_{\lambda_2}^{-1}(u),\lambda_2, \delta+\alpha_1\lambda_2+\alpha_2\lambda_1+\beta_1F_{\lambda_2}^{-1}(u)+\beta_2n)$, 
where $\lambda^{\ast}=\delta/(1-\alpha_1-\alpha_2)$ and $u\in (0,1)$. Note that we cannot define $\XX_t$ in the same way as in the first case, since otherwise it contradicts the independence assumption of $\{u_t\}$ sequence. Define the metric $\rho$ on $S=\mathbb{Z}_0\times[\lambda^{\ast},\infty)\times[\lambda^{\ast},\infty)$ as $\rho(\xx,\yy)=\sum_{i=1}^3 w_i|x_i-y_i|$, where $\xx=(x_i)_{i=1}^3, \yy=(y_i)_{i=1}^3$ and $w_i>0, i=1,2,3$. Take $\xx_1=(n_0,\lambda_1^0,\lambda_2^0):=(0,\lambda^{\ast},\lambda^{\ast})$, then for any $\xx=(n,\lambda_1,\lambda_2)$, we have
\begin{eqnarray*}
\E\rho(\XX_1(\xx),\XX_1(\xx_1))&=&\int_0^1|f_u(\xx)-f_u(\xx_1)|du\\
                                     &=&\beta_2 w_3|n-n^0|+w_3\alpha_2|\lambda_1-\lambda_1^0|\\
                                     &&+\{w_1+w_2+(\alpha_1+\beta_1)w_3\}|\lambda_2-\lambda_2^0|.
\end{eqnarray*}
Similarly to the first case, one needs to solve the inequality
\begin{eqnarray*}
(\alpha_2+\beta_2)(w_1+w_2)&\le& [r-(\alpha_1+\beta_1)](\alpha_2+\beta_2)w_3\\
&\le& r(w_1+w_2)[r-(\alpha_1+\beta_1)]
\end{eqnarray*}
for an $r\in (0,1)$ and a strictly positive triple $(w_1, w_2, w_3)$. This can be achieved if $\alpha_1+\alpha_2+\beta_1+\beta_2<1$, which implies the quadratic equation $r^2-(\alpha_1+\beta_1)r-(\alpha_2+\beta_2)=0$ has a root $r_+\in (0,1)$. The result hence follows by a simple induction.
\end{enumerate}

\subsection*{B.9. Proof of Theorem \ref{NonLinearAsympNormal}}
According to Theorem \ref{AsympNormal}, we only need to establish the identifiability of the model. Similar to the proof of Theorem \ref{LinearAsymp}, one can demonstrate that if $X_t(\theta)=X_t(\theta_0), P_{\theta_0}$-a.s. for some $t$, where $\theta_0=(\delta_0, \alpha_0,\beta_0,\beta_{1,0}, \ldots, \beta_{K,0})$, then 
\begin{eqnarray*}
&&(\beta-\beta_0)Y_{t-1} + \displaystyle\sum_{k=1}^K (\beta_k-\beta_{k,0})(Y_{t-1}-\xi_k)^+\\
&=&\delta_0-\delta+\alpha_0 X_{t-1}(\theta_0)-\alpha X_{t-1}(\theta)\in \sigma\{Y_{t-2}, Y_{t-3},\ldots\}.
\end{eqnarray*}
It follows that $\beta=\beta_0$ and $\beta=\beta_{k,0}, k=1,\ldots, K$. Similarly one can show that $\delta=\delta_0$ and $\alpha=\alpha_0$, hence $\theta=\theta_0$ which verifies the identifiability of the model.

\bibliographystyle{rss}
\bibliography{heng}
\end{document}